%% file: paper.tex
\documentclass[11pt]{amsart}
\usepackage[dvipsnames,usenames]{color}
\usepackage{hyperref}
\usepackage{graphicx}
\usepackage{epsfig}
\usepackage[latin1]{inputenc}
\usepackage{amsmath}
\usepackage{amsfonts}
\usepackage{amssymb}
\usepackage{amsthm}
\usepackage{amscd}
\usepackage{verbatim}
\usepackage{subfigure}
\usepackage{pinlabel}
\usepackage{stmaryrd}
\usepackage{enumerate, enumitem}
\usepackage{todonotes}
\usepackage{bm}
\usepackage{thmtools}
\usepackage{thm-restate}

\usepackage{tikz}
\usetikzlibrary{arrows}
\usetikzlibrary{decorations.pathreplacing}
\usepackage{verbatim}
\usetikzlibrary{cd}

\newtheorem{theorem}{Theorem}[section]

\newtheorem{lemma}[theorem]{Lemma}
\newtheorem{proposition}[theorem]{Proposition}

\newtheorem{corollary}[theorem]{Corollary}

\theoremstyle{definition}
\newtheorem{definition}[theorem]{Definition}

\theoremstyle{remark}
\newtheorem{remark}[theorem]{Remark}

\newcommand{\alphas}{\boldsymbol{\alpha}}
\newcommand{\betas}{\boldsymbol{\beta}}

\def\F{\mathbb{F}}
\def\N{\mathbb{N}}
\def\Q{\mathbb{Q}}

\def\Z{\mathbb{Z}}
\def\A{\mathbb{A}}
\def\B{\mathbb{B}}

\def\del{\partial}

\def\cC{\mathcal{C}}
\def\dD{\mathcal {D}}
\def\H{\mathcal{H}}
\def\cF{\mathcal{F}}
\def\cP{\mathcal{P}}
\def\cT{\mathcal{T}}
\def\rR{\mathcal R}

\def\Sym{\mathrm{Sym}}

\def\CFKi{\CFK^{\infty}}

\def\coker{\operatorname{coker}}
\def \gr {\operatorname{gr}}

\def\Cone{\operatorname{Cone}}

\def\ccdot {\! \cdot \!}

\def \wt{\widetilde}

\def\d{\partial}

\def\co{\colon}
\def\spinc{\textrm{Spin}^c}

\def\id{\textup{id}}
\def\Inv{\mathfrak{I}}

\def\im{\operatorname{im}}
\def\red{\textup{red}}
\def\conn{\textup{conn}}
\def\Cone{\operatorname{Cone}}

\newcommand{\bunderline}[1]{\underline{#1\mkern-2mu}\mkern2mu }

\def\du {\bar{d}}
\def\dl {\bunderline{d}}

\def\CF {\mathit{CF}}
\def\HF {\mathit{HF}}

\newcommand\HFp {\HF^+}

\newcommand \CFm {\CF^-}
\newcommand \HFm {\HF^-}

\newcommand \CFo {\CF^{\circ}}
\newcommand \HFo {\HF^{\circ}}
\def\HFred{\HF^{\operatorname{red}}}
\def\HFb{\mathbf{HF}^-}

\def\CFK{\mathit{CFK}}
\def\Ab{\mathbf{A}^-}
\def\Bb{\mathbf{B}^-}

\def\vb{\mathbf{v}}
\def\hb{\mathbf{h}}
\def\Fb{\mathbf{F}^-}
\def\Cb{\bm{\mathfrak{C}}}

\def\spinc {{\operatorname{spin^c}}}
\def\Spinc {{\operatorname{Spin^c}}}
\def\s{\mathfrak s}
\def\t{\mathfrak t}

\def\CFI {\mathit{CFI}}
\def\HFI {\mathit{HFI}}

\newcommand \CFIp {\CFI^+}
\newcommand \CFIm {\CFI^-}
\newcommand \HFIm {\HFI^-}

\newcommand \HFIinf {\HFI^{\infty}}
\newcommand \CFIo {\CFI^{\circ}}
\newcommand \HFIo {\HFI^{\circ}}
\def\HFred{\HF_{\operatorname{red}}}
\def\inv{\iota}
\def\iotaperp{\iota^\perp_f}

\def\Hom{\mathrm{Hom}}
\def\Tor{\mathrm{Tor}}

\author[Kristen Hendricks]{Kristen Hendricks}
\thanks{The first author was partially supported by NSF grant DMS-1663778.}
\address{Department of Mathematics, Michigan State University, East Lansing, MI, 48824}
\email{hendricks@math.msu.edu}

\author[Jennifer Hom]{Jennifer Hom}
\thanks{The second author was partially supported by NSF grant DMS-1552285 and a Sloan Research Fellowship.}
\address {School of Mathematics, Georgia Institute of Technology, Atlanta, GA 30332}
\email{hom@math.gatech.edu}

\author[Tye Lidman]{Tye Lidman}
\thanks{The third author was partially supported by NSF grant DMS-1709702.}
\address{Department of Mathematics, North Carolina State University, Raleigh, NC, 27607}
\email{tlid@math.ncsu.edu}

\numberwithin{equation}{section}

\title{Applications of involutive Heegaard Floer homology}

\begin{document}

\begin{abstract} We use Heegaard Floer homology to define an invariant of homology cobordism.  This invariant is isomorphic to a summand of the reduced Heegaard Floer homology of a rational homology sphere equipped with a spin structure and is analogous to Stoffregen's connected Seiberg-Witten Floer homology.  We use this invariant to study the structure of the homology cobordism group and, along the way, compute the involutive correction terms $\du$ and $\dl$ for certain families of three-manifolds.    
\end{abstract}

\maketitle

\input{introduction}

\input{background}

\input{connectedinvolutive}

\input{connectedproperties}

\input{applications}

\input{surgeries}

\input{HFconncomputations}

\bibliographystyle{amsalpha}
\bibliography{bib}

\end{document}

%% file: introduction.tex
\section{Introduction}\label{sec:intro}

The study of homology cobordism, or when two manifolds cobound a homology cylinder, has been a motivating structure in geometric topology for several decades.  Most recently, Manolescu used an invariant of homology cobordism to disprove the high-dimensional triangulation conjecture \cite{ManolescuTriangulation}.  While the result applies to triangulating manifolds of dimensions at least five, the invariant is for spin rational homology three-spheres and smooth homology cylinders between them.   

The key idea in constructing Manolescu's invariant is defining a $Pin(2)$-equivariant Seiberg-Witten Floer homology for a three-manifold equipped with a self-conjugate $\spinc$ structure.  Building on this, Stoffregen \cite{Stoffregen} constructed a more refined invariant, the connected Seiberg-Witten Floer homology, which takes the form of a graded module over $\F[U]$.  Very roughly, this is defined as the part of the $S^1$-equivariant Seiberg-Witten Floer homology which consists of solutions to the Seiberg-Witten equations that interact with the reducible in an essential way that respects the $Pin(2)$-symmetry.  One disadvantage of Stoffregen's construction is that it passes through Manolescu's Seiberg-Witten Floer homotopy type \cite{ManolescuSWF}, which can be rather difficult to compute explicitly.  

In the current article, we define an analog of Stoffregen's connected Seiberg-Witten Floer homology in the setting of Heegaard Floer homology and use this to further study the structure of the homology cobordism group.  

\begin{theorem}\label{thm:HFconn}
Let $Y$ be a rational homology three-sphere and $\s$ a spin structure on $Y$. There is a spin rational homology cobordism invariant, $\HF_\conn(Y, \s)$, called the {\em connected Heegaard Floer homology}, taking values in isomorphism classes of absolutely-graded $\F[U]$-modules. Moreover, $\HF_\conn(Y, \s)$ is isomorphic to a summand of $\HF_\red(Y, \s)$.
\end{theorem}

In order to define this invariant, we make use of work of the first author and Manolescu \cite{HM:involutive}, in which they use the $\spinc$-conjugation symmetry in Heegaard Floer homology to produce {\em involutive Heegaard Floer homology} for a pair $(Y,\s)$.  (Recall that on a rational homology sphere, self-conjugate $\spinc$ structures correspond precisely with spin structures.)  In this case, the involutive package contains two numerical homology cobordism invariants: $\du(Y,\s)$ and $\dl(Y,\s)$.  Using $\HF_\conn$, we describe the asymptotic behavior of $\dl$ and $\du$ under connected sums; this result is the Heegaard Floer-analog of \cite[Theorem 1.3]{Stoffregenconnectedsum}, which concerns the Manolescu invariants $\alpha$, $\beta$, and $\gamma$.

\begin{theorem}\label{thm:asymptotics}
Let $Y$ be a rational homology three-sphere and $\s$ a spin structure on $Y$.  Then
\[ \lim_{n\rightarrow \infty} \frac{\dl(\#_n(Y, \s))}{n} =  \lim_{n\rightarrow \infty} \frac{\du(\#_n(Y, \s))}{n} = d(Y,\s). \]
\end{theorem}

Let $\Theta^3_\Z$ be the three-dimensional homology cobordism group, i.e., 
\[ \Theta^3_\Z = \{ \Z \textup{-homology spheres} \} / \Z \textup{-homology cobordism}. \]
This group is also often called $\Theta^3_H$. We next use $\HF_\conn$ to define a filtration on $\Theta^3_\Z$. For notation, let $\mathcal T_a(n) = \F_{(a)}[U]/U^{n}$, where $\gr(1) = a$.  Let $\cP$ denote the set of subsets of $\N$. Note that $\cP$ is a partially ordered set, with the order induced by inclusion, i.e., given $P_1, P_2 \in \cP$, we have that $P_1 \leq P_2$ if $P_1 \subseteq P_2$. For $P \in \cP$, define
\[ \cF_P = \{ [Y] \mid \HF_\conn(Y) \cong \bigoplus_{i=1}^N \cT_{a_i}(n_i) , \ n_i \in P \}, \]
Here, $(n_i)_{i=1}^N$ is any finite sequence of elements in $P$.

\begin{restatable}{proposition}{filtered}
\label{prop:filtered}
The group $\Theta^3_\Z$ is filtered by $\cP$, i.e., 
\begin{enumerate}
	\item $\cF_P$ is a subgroup for each $P \in \cP$,
	\item if $P_1 \leq P_2$, then $\cF_{P_1} \subseteq \cF_{P_2}$.
\end{enumerate}
\end{restatable}

From the filtration $\cP$, we are able to reprove Furuta's theorem that $\Theta^3_\Z$ contains a subgroup of infinite rank \cite{Furuta}.  
\begin{restatable}{theorem}{infinitelygenerated}
\label{thm:-1surgerytorus}
The manifolds $S^3_{-1}(T_{2, 4n+1})$ are linearly independent in $\Theta^3_\Z$.
\end{restatable}
This was also recently done with involutive Heegaard Floer homology by \cite{DaiManolescu} (see also \cite{Stoffregenconnectedsum} for a proof using Seiberg-Witten theory).  Note that the knots in the above theorem have arbitrarily large genus. In fact, if we restrict to surgery on knots with bounded genus, the subgroup generated by these manifolds will never be all of $\Theta^3_\Z$, even if we allow fractional surgeries.

\begin{restatable}{theorem}{genus}
\label{thm:g(K)<N}
Fix $N > 0$.  Let $\Theta^3_N$ denote the subgroup of $\Theta^3_\Z$ generated by $\{ S^3_{1/n}(K) \mid g(K) <N, n \in \Z \}$. Then $\Theta^3_N$ is a proper subgroup of $\Theta^3_\Z$.
\end{restatable}

As seen from the construction of $\cP$, the $U$-action on $\HF_\conn$ can be used to obtain significant information.  We therefore define an invariant that measures the nilpotence of this action.
\begin{definition}\label{def:omega}
Let $(Y,\s)$ be a spin rational homology sphere. Define
\[ \omega(Y,\s) = \min \{ n \mid U^n \HF_\conn(Y,\s)=0 \}. \]
\end{definition}

\begin{proposition}\label{prop:omegabounds} 
Let $(Y,\s)$ be a spin rational homology sphere. Then
\begin{enumerate}
	\item $\frac{1}{2}(d(Y,\s)-\dl(Y,\s)) \leq \omega(Y,\s)$,
	\item $\frac{1}{2}(\du(Y,\s)-d(Y,\s)) \leq \omega(Y,\s)$.
\end{enumerate}
\end{proposition}

The proof of Theorem \ref{thm:-1surgerytorus} relies on the following calculation. Throughout, we will use the standard identification between $\spinc$ structures on $S^3_{\frac{p}{q}}(K)$ and classes $[i]$ in $\Z/p\Z$ given in \cite[Section 2.4]{OSinteger}.  We also will use the concordance invariant $V_0$ from \cite{RasmussenThesis} (see \cite{NiWu} for the current notation).  Recall that an \emph{L-space knot} $K \subset S^3$ is a knot that admits a positive L-space surgery.

\begin{theorem}\label{thm:-nLspace}
Let $Y=S^3_{-n}(K)$ where $K$ is an L-space knot and $n$ is a positive integer. Then
\[ \HF_\conn(Y, [0]) = \cT_{-1}(V_0), \]
and $d(Y, [0]) = \du (Y, [0]) = -d(L(n,1), [0])$ and $\dl(Y, [0]) = -2V_0-d(L(n,1), [0])$.
\end{theorem}
\noindent For a complete description of the involution $\iota_*$ on $\HFm(Y)$, see Proposition \ref{prop:surgeryinvolution}.

Here our assignment $\Z/n\Z \simeq \Spinc(L(n,1))$ is via viewing $L(n,1)$ as $n$ surgery on the unknot. While Theorem~\ref{thm:-nLspace} does not hold for arbitrary knots, in general we have the following inequality for the invariant $\dl$.
\begin{restatable}{theorem}{surgerygeneral}\label{thm:-ngeneral}
Let $K$ be a knot in $S^3$ and $n$ a positive integer. Then 
\[ \dl(S^3_{-n}(K), [0]) \geq - 2V_0(K) - d(L(n,1), [0]). \]
\end{restatable}

Using Theorem~\ref{thm:-nLspace}, we give a complete computation of $\HF_{\conn}$ for connected sums of $-1$ surgeries on L-space knots. 

\begin{theorem} \label{thm:connect-sums} Let $K_1, \cdots, K_m$ be concordant to L-space knots, ordered so that $V_0(K_1)\geq V_0(K_2)\geq \cdots \geq V_0(K_m))$. Let $a_i = \sum_{j=1}^{i} V_0(K_i)$, with $a_0=0$. Then
\[
\HF_\conn (\#_{i=1}^m S^3_{-1}(K_i)) = \bigoplus_{i=1}^{m} \mathcal T_{i-2-2a_{i-1}}(V_0(K_i)).
\]
\end{theorem}

Theorem~\ref{thm:connect-sums} allows us to easily construct homology spheres which are not homology cobordant to Seifert fibered spaces, first appearing in \cite[Corollary 1.11]{Stoffregen}.   

\begin{corollary} Let $K_1, \cdots, K_n$ and $K'_1, \cdots, K'_m$ be knots which are concordant to non-trivial L-space knots such that $n\neq m$. Then $\#_{i=1}^n S^3_{-1}(K_i)$ is not homology cobordant to $\#_{i=1}^m S^3_{-1}(K'_i)$. Moreover, for $n\geq 2$, $\#_{i=1}^n S^3_{-1}(K_i)$ is not homology cobordant to any Seifert fibered space or any surgery on an L-space knot.
\end{corollary}
\begin{proof}
The reduced Floer homology of any Seifert fibered rational homology sphere is only supported in a fixed parity of gradings \cite[Theorem 3.3]{OSseifert}.  The same is true for non-zero rational surgery on an L-space knot. The result now follows from Theorem~\ref{thm:connect-sums}, since $V_0(K) > 0$ for any knot concordant to a non-trivial L-space knot \cite[Theorem 1.2]{OSlens}.  
\end{proof}

\begin{remark}
The above corollary should be compared to \cite[Theorem 1.2]{DaiStoffregen}; indeed, up to a grading shift, for $K$ an L-space knot, $S^3_{-1}(K)$ is locally equivalent to Dai-Stoffregen's $Y_{V_0}$.
\end{remark}

Finally, we use the connected Floer homology to show that certain elements of the homology cobordism group are infinite order, analogous to work of Lin-Ruberman-Saviliev \cite[Theorems C and D]{LRS} for certain manifolds with non-vanishing Rokhlin invariant.  Recall that an integer homology sphere $Y$ is said to be \emph{$d$-negative} if $\gr(x)< d(Y)$ for every non-trivial grading homogeneous element $x \in \HF_{\red}(Y)$.

\begin{restatable}{theorem}{infiniteorder} \label{thm:infiniteorder} Let $Y$ be an integer homology sphere such that $Y$ is $d$-negative and $\dl(Y) < d(Y)$. Then $\dim_{\F}(\HF_\conn(\#_nY))\geq 1$ for all $n \neq 0$. In particular, $[Y]$ has infinite order in $\Theta^3_\Z$.  \end{restatable}

Without any additional assumptions, in the special case that $\dim_{\F}\HF_\conn(Y) = 1$, we are able to prove a similar statement.
\begin{restatable}{theorem}{infiniterankone}\label{thm:infiniterankone}
Let $Y$ be an integer homology sphere.  If $\dim_{\F} \HF_\conn(Y) = 1$, then $Y$ is of infinite order in $\Theta^3_\Z$.
\end{restatable}

\begin{remark} Because $\HF_\conn$ is a spin homology cobordism invariant, Theorems \ref{thm:infiniteorder} and \ref{thm:infiniterankone} are equally true if $(Y,\s)$ is a $\Z_2$-homology sphere with its unique spin structure, and we consider its order in the $\Z_2$-homology cobordism group $\Theta^3_{\Z_2}$, or in the quotient $\Theta^3_{\Z_2}/\Theta^3_L$ of the group $\Theta^3_{\Z_2}$ by the group $\Theta^3_L$ generated by the Heegaard Floer L-spaces which are also $\Z_2$-homology spheres. (In this latter case, modding out by $\Theta^3_L$ has the effect of replacing the absolute $\Q$-grading on Heegaard Floer homology by a relative $\Z$-grading.)\end{remark}

We conclude with a general computation of the connected homology of homology spheres $Y$ whose homology $\HFm(Y)$ is represented by the $\F[U]$ module $\mathbb H^-(M)$ associated to a symmetric graded root $M$ (including the Seifert fibred spaces) up to an appropriate grading shift. Our computation proceeds quickly from recent computations of Dai and Manolescu \cite{DaiManolescu}. In this paper, the authors associate to any symmetric graded root $M$ a preferred monotone subroot $M'$. (We review this construction in Section \ref{subsec:gradedroots}.) Our result is the following.

\begin{theorem} \label{thm:gradedroots} Let $Y$ be an integer homology sphere with the property that $\HFm(Y) \simeq \mathbb H^-(M)$ for some graded root $M$. Then the connected homology of $Y$ is the $U$-torsion submodule of $\mathbb H^-(M')$, shifted upward in degree by $1$.\end{theorem}

We will give a more precise technical statement of the above theorem in Section \ref{subsec:gradedroots}. We have the following corollary of Theorem \ref{thm:gradedroots}, which shows that the connected Heegaard Floer homology of Seifert fibered homology spheres agrees with Stoffregen's computation of the connected Seibeg-Witten Floer homology \cite[Corollary 1.7]{Stoffregen}.

\begin{corollary}\label{cor:gradedroot}
Let $Y$ be an integer homology sphere such that $\HFm(Y) \simeq \mathbb H^-(M)$ for some graded root $M$ with
$ \HFred (Y) \cong \bigoplus_{i=1}^N \big(\cT_{a_i} (n_i) \big)^{k_i}$, where $(a_i, n_i) \neq (a_j, n_j)$ if $i \neq j$.
Then
\[ \HF_\conn(Y) \cong \bigoplus_{i=1}^N \big(\cT_{a_i} (n_i) \big)^{k'_i},\]
where
\[
k'_i =
\begin{cases}
0 & \text{ if } k_i \text{ even} \\
1 & \text{ if } k_i \text{ odd}.
\end{cases}
\]
\end{corollary}


\subsection{Organization} This paper is organized as follows.  In Section~\ref{sec:background} we give the necessary background on involutive Heegaard Floer homology.  Section~\ref{sec:connected} is where we define the connected Heegaard Floer homology.  Several properties of this invariant are given in Section~\ref{sec:properties}.  Applications of connected Floer homology to homology cobordism are given in Section \ref{sec:applications}. The involutive structure on the Floer homology of certain surgeries is then computed in Section~\ref{sec:surgeries}, which leads to the promised computations of the connected Heegaard Floer homologies of connected sums in Section~\ref{subsec:connectsums}. In Section \ref{subsec:gradedroots}, we compute the connected homology of graded roots.  

\section*{Acknowledgements}
We thank John Etnyre, Ciprian Manolescu, and Matt Stoffregen for helpful conversations.

%% file: background.tex
\section{Background on involutive Heegaard Floer homology}\label{sec:background}

In this section we briefly review involutive Heegaard Floer homology and the group of $\iota$-complexes, following \cite{HM:involutive} and \cite{HMZ:involutive-sum}. We assume that the reader is familiar with ordinary Heegaard Floer homology, as in \cite{OS3manifolds1, OS3manifolds2, OS:four, OSabsgr}. Throughout, we work over coefficients in the field $\F = \Z/2\Z$, and restrict ourself to the case of a rational homology sphere $Y$ with a spin structure $\s$. In this case Ozsv{\'a}th and Szab{\'o}'s Heegaard Floer groups $\HFo(Y,\s)$, with $\circ \in \{+,-,\widehat{\ }, \infty\}$, are modules over $\F[U]$ with an absolute $\Q$-grading.

Before going further, it will be helpful to state our grading conventions. We let $\HFp(S^3) = \F[U, U^{-1}]/U\F[U]$ with $\gr(1) = 0$, whereas $\HFm(S^3) = \F[U]$ with $\gr(1) = -2$. The module $\HFred(Y, \s)$ takes gradings as a quotient of $\HFp(Y, \s)$, or equivalently, as a non-canonical submodule of $\HFp(Y, \s)$. The $U$-torsion submodule $R$ of $\HFm(Y)$ is also isomorphic to $\HFred(Y)$. In particular, with this choice of conventions, $\HFred(Y) \cong R[-1]$. (Here and throughout, given a graded module $M$, we write $M[-n]$ to denote the module $M$ with all gradings shifted upward by $n$.) When we introduce $\HF_{\conn}(Y)$ in Section \ref{sec:connected}, this module will take gradings as a submodule of $\HFred(Y) \subset \HFp(Y)$. 

\subsection{Involutive Heegaard Floer homology} We now consider involutive Heegaard Floer homology. Let $\H = (H, J)$ be a Heegaard pair for $Y$ consisting of a pointed Heegaard diagram $H = (\Sigma, \alphas, \betas, z)$ and a family of almost complex structures $J$ on the symmetric product $\Sym^g(\Sigma)$. In \cite{HM:involutive}, the first author and Manolescu construct a grading preserving chain map $\iota$, called the conjugation involution, on the Heegaard Floer chain complexes, and prove that $\iota^2$ is chain homotopic to the identity map. This involution is constructed as follows. There is a conjugate Heegaard pair to $\H$ given by $\bar{\H} = (\bar{H}, \bar{J})$, with $\bar{H} = (-\Sigma, \betas, \alphas, z)$. This is also a Heegaard pair for $Y$, and there is a chain isomorphism $\eta \co \CFo(\H,\s) \rightarrow \CFo(\bar{\H}, \s)$. Using work of Juh\'asz and Thurston \cite{JuhaszThurston}, one can then choose a sequence of Heegaard moves and changes to the family of almost complex structures from $\bar{\H}$ to $\H$, giving a chain homotopy equivalence $\Phi(\bar{\H}, \H) \co \CFo(\bar{\H},\s) \rightarrow \CFo(\H, \s)$. Then the map $\iota$ is the composition of these maps. That is, we have
\[
\iota = \Phi(\bar{\H}, \H) \circ \eta \co \CFo(\H, \s) \rightarrow \CFo(\H, \s).
\]
For each of the four flavors of Heegaard Floer homology, the involutive Heegaard Floer chain complex $\CFIo(\H,\s)$ is then the mapping cone
\[\CFIo(\H,\s) = \Cone(\CFo(\H,\s) \xrightarrow{Q(1+\iota)} Q\cdot\CFo(\H,\s)[-1]).
\]
Here $Q$ is a formal variable with $\deg(Q)=-1$, so that if $x \in \CFo(\H,\s)$ has $\gr(x)=r$, in the complex $\CFIo(\H,\s)$ we have $\gr(x) = r+1$ and $\gr(Qx) = r$. The resulting chain complex is a module over $\rR = \F[U,Q]/(Q^2)= H^*(B\Z_4; \Z_2)$. The \emph{involutive Heegaard Floer homology} $\HFIo(Y,\s)$ is the homology of this chain complex; its isomorphism type is an invariant of $(Y,\s)$. For this paper, we will sometimes abuse notation slightly and refer to the chain complex $\CFm(Y, \s)$ as having an involution $\iota$. This denotes an appropriate representative of the chain homotopy class of the complex with its map $\iota$, which is only an involution up to homotopy.

It is immediate from the construction that involutive Heegaard Floer homology fits into an exact triangle
\begin{equation}
\label{pic:exact}
\begin{tikzpicture}[baseline=(current  bounding  box.center)]
\node(1)at(0,0){$\HFIo(Y,\s)$};
\node(2)at (-2,1){$\HFo(Y, \s)$};
\node(3)at (2,1){$Q \ccdot \HFo(Y,\s)[-1]$};
\path[->](2)edge node[above]{$Q(1+ \inv_*)$}(3);
\path[->](3)edge (1);
\path[->](1)edge(2);
\end{tikzpicture}
\end{equation}
of $U$-equivariant maps \cite[Proposition 4.6]{HM:involutive}. 

Involutive Heegaard Floer homology also behaves well under orientation reversal, as follows.

\begin{proposition}[{\cite[Proposition 4.4]{HM:involutive}}]\label{prop:orientation}
 If $\s$ is a spin structure on $Y$, there is an isomorphism
\[
\CFIp_r(Y, \s) \rightarrow \CFI_{-}^{-r-1}(-Y,\s),
\]
where $\CFI_-$ denotes the cochain complex dual to $\CFIm$ over $\F$.
\end{proposition}

The first author and Manolescu extract two numerical invariants of homology cobordism from $\HFI$. Phrased in terms of $\HFIm$, these invariants are as follows.

\begin{definition} Let $(Y,\s)$ be a rational homology sphere with a spin structure. Then the involutive correction terms are
\[ \dl(Y,\s) = \max \{r \mid \exists \ x \in \HFIm_r(Y, \s), \forall \ n, \ U^nx\neq 0 \ \text{and} \ U^nx \notin \operatorname{Im}(Q)\} + 1 \]
and
\[ \du(Y,\s) = \max \{r \mid \exists \ x \in \HFIm_r(Y,\s),\forall \ n, U^nx\neq 0; \exists \ m\geq 0 \ \operatorname{s.t.} \ U^m x \in \operatorname{Im}(Q)\} +2. \]

\noindent Equivalently, $\du(Y,\s)=r+2$ is two more than the maximal grading $r$ such that $r \equiv d(Y,\s)$ modulo $2\Z$ and the map $i \co \HFIm_r(Y,\s) \rightarrow \HFIinf_r(Y,\s)$ is nonzero, and $\dl(Y,\s)=q+1$ is one more than the maximal grading $q$ such that $q \equiv d(Y,\s)+1$ modulo $2\Z$ and $i \co \HFIm_q(Y,\s) \rightarrow \HFIinf_q(Y,\s)$ is nonzero.
\end{definition}

It follows from the long exact sequence (\ref{pic:exact}) that the correction terms satisfy
\[\dl(Y,\s) \leq d(Y,\s) \leq \du(Y,\s)\]
and from Proposition \ref{prop:orientation} that
\[\dl(Y, \s) = -\du(-Y,\s).\]
Furthermore, the correction terms are invariants of $\Z_2$ homology cobordism, and therefore descend to set maps
\[
\dl, \du \co \Theta^3_{\Z_2} \rightarrow \Q.
\]

In \cite{HMZ:involutive-sum}, Manolescu, Zemke, and the first author show that involutive Heegaard Floer homology obeys a convenient connected sum formula, as follows. Recall from \cite[Theorem 1.5]{OS3manifolds2} that Ozsv{\'a}th and Szab{\'o} give a chain homotopy equivalence
\begin{equation}
\label{eq:KunnethCF}
\CFm(Y_1 \# Y_2, \s_1 \# \s_2) \simeq \CFm(Y_1,\s_1) \otimes_{\F[U]} \CFm(Y_2, \s_2)[-2].
\end{equation}

The grading shift is necessary since we take $\HFm(S^3) = \F[U]$ with $\gr(1)=-2$. With respect to this chain homotopy equivalence, we have the following.

\begin{proposition}[{\cite[Theorem 1.1]{HMZ:involutive-sum}}]\label{prop:ConnSum}
Suppose $Y_1$ and $Y_2$ are three-manifolds equipped with spin structures $\s_1$ and $\s_2$. Let $\iota_1$ and $\iota_2$ denote the conjugation involutions on the Floer complexes $\CFm(Y_1, \s_1)$ and $\CFm(Y_2,\s_2)$. Then, under the equivalence \eqref{eq:KunnethCF}, the conjugation involution $\iota$ on $\CFm(Y_1 \# Y_2,$ $ \s_1 \# \s_2)$ is chain homotopy equivalent, over the ring $\F[U]$, to $\iota_1 \otimes \iota_2.$ \end{proposition}

As a corollary, one obtains the following behavior of the correction terms under connected sum.

\begin{proposition}[{\cite[Proposition 1.3]{HMZ:involutive-sum}}] \label{prop:inequalities} Let $(Y_1,\s_1)$ and $(Y_2,\s_2)$ be rational homology spheres equipped with spin structures. Then we have
\[
\dl(Y_1,\s_1) + \dl(Y_2,\s_2) \leq \dl(Y_1\#Y_2, \s_1\#\s_2) \leq \dl(Y_1,\s_1) + \du(Y_2,\s_2) \leq \du(Y_1 \# Y_2, \s_1\#\s_2) \leq \du(Y_1,\s_1) + \du(Y_2,\s_2). \]
\end{proposition}

Explicit computations of involutive Heegaard Floer homology have been done for large surgeries on L-space and thin knots by the first author and Manolescu \cite{HM:involutive}, for Seifert fibered spaces by Dai and Manolescu \cite{DaiManolescu}, and for certain connected sums of these examples \cite{HMZ:involutive-sum, DaiManolescu, DaiStoffregen}.

\subsection{The group of $\iota$-complexes} In light of the connected sum formula of Proposition \ref{prop:ConnSum}, one can define a group out of abstract $\F[U]$-complexes with involutions satisfying the same structural properties as $(\CFm(Y,\s), \iota)$ mentioned above. Throughout this section, $\simeq$ will denote a chain homotopy of maps over $\F[U]$.  We recall the following definition.

\begin{definition}[{\cite[Definition 8.5]{HMZ:involutive-sum}}] An {\em $\iota$-complex} $\cC = (C, \iota)$ consists of the following data:

\begin{itemize}
\item A finitely-generated, $\mathbb Q$-graded, free chain complex $C$ over $\F[U]$ such that there is some $\tau \in \mathbb Q$ such that $C$ is supported in gradings differing from $\tau$ by integers (i.e., so that $C$ is relatively $\Z$-graded and absolutely $\Q$-graded) and furthermore there is a relatively graded isomorphism 
\[
U^{-1}H_*(C) \cong \F[U,U^{-1}].
\]
\item A grading preserving chain map $\iota \co C \rightarrow C$ such that $\iota^2 \simeq \id$.  
\end{itemize}

\end{definition}

\noindent Here by $U^{-1}H_*$ we mean the result of localizing $H_*$ at $U$. The definition above slightly extends the original, which was only for absolutely $\Z$-graded complexes; this formulation first appears in \cite[Definition 2.1]{DaiManolescu}.

There are two natural notions of equivalence of $\iota$-complexes. Firstly, one can consider chain homotopy equivalence.

\begin{definition}
Two $\iota$-complexes $\cC = (C, \iota)$ and $\cC'=(C', \iota')$ are \emph{homotopy equivalent}, denoted $\cC \simeq \cC'$, if there exist grading-preserving chain maps
\begin{align*}
	f &\co C \rightarrow C' \\
	g &\co C' \rightarrow C
\end{align*}
such that
\begin{enumerate}
	\item $f \circ \iota \simeq \iota' \circ f$ and $g \circ \iota' \simeq \iota \circ g$,
	\item $g \circ f \simeq \id_C$ and $f \circ g \simeq \id_{C'}$.
\end{enumerate}
\end{definition}

However, it is slightly more useful to consider a weaker relation, called local equivalence.

\begin{definition} Two $\iota$-complexes $\cC = (C, \iota)$ and $\cC'=(C', \iota')$ are said to be \emph{locally equivalent} if there exist grading-preserving chain maps
\begin{align*}
f &\co C \rightarrow C' \\
g &\co C' \rightarrow C
\end{align*}
such that
\begin{enumerate}
\item $f \circ \iota \simeq \iota' \circ f$ and  $g \circ \iota' \simeq \iota \circ g$,
\item $f$ and $g$ induce isomorphisms on $U^{-1}H_*$.
\end{enumerate}
\end{definition}

The importance of this second notion of equivalence lies in the following lemma.

\begin{lemma}[{\cite[Proof of Theorem 1.8]{HMZ:involutive-sum}}]\label{lem:homology-cobordism} Suppose that $(Y_1, \s_1)$ and $(Y_2, \s_2)$ are related by a spin rational homology cobordism. Then $(\CFm(Y_1,\s_1), \iota_1)$ is locally equivalent to $(\CFm(Y_2, \s_2), \iota_2)$ as $\iota$-complexes.
\end{lemma}

With this in mind, one can consider the group $\mathfrak I_{\Q}$ consisting of $\iota$-complexes $\cC=(C,\iota)$ modulo local equivalence, with multiplication given by
\[
\cC \otimes \cC' = (C \otimes_{\F[U]} C'[-2], \iota \otimes \iota').
\]
The identity element is $[(\F[U], \id)]$, with $\gr(1)=-2$. The inverse of $[(C,\iota)]$ is given by $ \cC^* = [(C^*, \iota^*)]$, where $C^* = \Hom_{\F[U]}(C, \F[U])$ and $\iota^*$ is the dual map to $\iota$. Gradings in the complex $C^*$ are handled as follows: if $\bf{S}$ is a set of generators for $C$ over $\F[U]$, for each $ x \in \bf S$ the dual generator $x^*$ has grading $-\gr(x) - 4$.

\begin{proposition}[{\cite[Proposition 8.8]{HMZ:involutive-sum}}] With the definitions above, $\mathfrak I_{\Q}$ is a well-defined abelian group.
\end{proposition}

Once again, this is a slight extension of the original to allow for $\Q$-gradings; if one restricts to absolutely $\Z$-graded complexes, one obtains the group $\mathfrak I$ of \cite{HMZ:involutive-sum}. The extension first appeared in \cite{DaiManolescu}.

Given an $\iota$-complex, one can define the correction terms analogously to the case of $\CF^-(Y,\s)$ and $\CFIm(\H, \s)$. In this paper, we will call these correction terms $d(\cC)$, $\dl(\cC)$, and $\du(\cC)$. (Of course, the first correction term only depends on $C$ and not $\iota$.) In particular, we have
\[d(\cC) = \max \{r \mid \exists \ x \in H_r(C), \forall \ n, \ U^nx \neq 0 \} +2, \]
\[ \dl(\cC) = \max \{r \mid \exists \ x \in \HFIm_r(\cC), \forall \ n, \ U^nx\neq 0 \ \text{and} \ U^nx \notin \operatorname{Im}(Q)\} + 1, \]
and
\[ \du(\cC) = \max \{r \mid \exists \ x \in \HFIm_r(\cC),\forall \ n, U^nx\neq 0; \exists \ m\geq 0 \ \operatorname{s.t.} \ U^m x \in \operatorname{Im}(Q)\} +2. \]
These correction terms descend to give functions
\[
\dl, d, \du \co \mathfrak I_{\Q} \rightarrow \Q.
\]
\noindent So in total one has functions
\[
\Theta^3_{\Z_2} \rightarrow \mathfrak I_{\Q} \xrightarrow{ \dl, d, \du}  \Q
\]
\noindent where the first step is via $(Y,\s) \mapsto [(\CF^-(Y, \s), \iota)]$. If we consider only integer homology spheres, this becomes
\[
\Theta^3_{\Z} \rightarrow \mathfrak I \xrightarrow{ \dl, d, \du} 2\Z.
\]

The group $\mathfrak I$ has been subsequently studied by Dai and Manolescu, who computed the local equivalence classes of $(\CFm(Y,\s), \iota)$ for $Y$ an almost-rational plumbed manifold \cite{DaiManolescu}, and by Dai and Stoffregen, who studied linear relationships between such manifolds \cite{DaiStoffregen}.  

%% file: connectedinvolutive.tex
\section{Connected Heegaard Floer homology}\label{sec:connected}

In this section, we define connected Heegaard Floer homology and prove Theorem \ref{thm:HFconn}. Our approach is an adaptation of Section 2.5 of \cite{Stoffregen}.

\begin{definition}\label{def:good}
Let $\cC= (C, \iota)$ be an $\iota$-complex. A grading preserving chain map 
\[ f \co C \rightarrow C \]
is a \emph{self-local equivalence} if
\begin{enumerate}
	\item $f \circ \iota \simeq \iota \circ f$,
	\item $f$ induces an isomorphism on $U^{-1}H_*$.
\end{enumerate}
\end{definition}

Recall that a \emph{preorder} on a set $S$ is binary relation $\lesssim$ on $S$ that is reflexive and transitive. We may define a preorder on the set of self-local equivalences of $\cC$ by $f \lesssim g$ if $\ker f \subseteq \ker g$. Note that we have a preorder rather than a partial order because $f \co C \rightarrow C$ is not uniquely determined by its kernel. We say that $f$ is \emph{maximal} if for $g$ a self-local equivalence, $g \gtrsim f$ implies $g \lesssim f$.  Roughly, the connected homology will be the torsion submodule of the homology of the image of a maximal self-local equivalence. The rest of this section is dedicated to showing that this is well-defined.  

We begin with the following technical lemma.

\begin{lemma}\label{lem:locally-equiv-isom}
Let $\cC=(C, \iota)$ and $\cC'=(C', \iota')$ be $\iota$-complexes. Suppose
\begin{align*}
	F &\co C \rightarrow C'
\end{align*}
is a chain complex isomorphism such that $F \circ \iota \simeq \iota' \circ F$. Then $\cC$ and $\cC'$ are homotopy equivalent as $\iota$-complexes.
\end{lemma}

\begin{proof}
By assumption, $F \circ \iota \simeq \iota' \circ F$ via some homotopy $H \co C \rightarrow C'$. Then it is straightforward to verify that $F^{-1} \circ \iota' \simeq \iota \circ F^{-1}$ via $F^{-1} \circ H \circ F^{-1} \co C' \rightarrow C$. Now $F$ and $F^{-1}$ provide the desired homotopy equivalence.
\end{proof}

We now prove several lemmas regarding maximal self-local equivalences.

\begin{lemma}
Maximal self-local equivalence always exist.
\end{lemma}

\begin{proof}
The set of self-local equivalences is nonempty (the identity is a self-local equivalence) and finite (since $C$ is finitely generated over $\F[U]$ hence finite dimensional as an $\F$-vector space in each grading). Therefore, maximal self-local equivalence always exist.
\end{proof}

\begin{lemma}\label{lem:isomim}
If $f, g \co C \rightarrow C$ are maximal self-local equivalences of $\cC = (C, \iota)$, then $f|_{\im g} \co \im g \rightarrow \im f$ is an isomorphism of chain complexes. 
\end{lemma}

\begin{proof}
The composition $g \circ f \co C \rightarrow C$ is a self-local equivalence and $\ker (g \circ f) \supseteq \ker f$. Since $f$ is maximal, $\ker (g \circ f) = \ker f$ and $g|_{\im f}$ is injective. Similarly, $f|_{\im g}$ is injective. Then we have injective grading preserving $\F[U]$-equivariant chain maps between finitely generated chain complexes over $\F[U]$, thus the maps are isomorphisms.
\end{proof}

\begin{lemma}\label{lem:splitcomplex}
If $f$ is a maximal self-local equivalence of $\cC=(C, \iota)$, then $C$ is isomorphic to the sum $\im f \oplus \ker f$.
\end{lemma}

\begin{proof}
By Lemma \ref{lem:isomim}, the map $f|_{\im f}$ is injective. Then a standard algebra argument shows that $C \cong \im f \oplus \ker f$. Namely, given $(x,y) \in \im f \oplus \ker f$ and $z \in C$, the maps
\begin{align*}
(x, y) &\mapsto (f|_{\im f})^{-1}(x)+y \\
z &\mapsto (f(z), z+(f|_{\im f})^{-1} \circ f(z))
\end{align*}
provide the desired isomorphism.
\end{proof}

For $f$ a maximal self-local equivalence, define $\iota_f\co \im f \to \im f$ by $f\circ \iota \circ (f|_{\im f})^{-1}$ and $\iotaperp \co \ker f \to \ker f$ by $(1 + (f|_{\im f})^{-1}\circ f ) \circ \iota$.
\begin{lemma}\label{lem:iotasplit}
If $f$ is a maximal self-local equivalence of $\cC=(C, \iota)$, then $(\im f, \iota_f) \oplus (\ker f, \iotaperp)$ is an $\iota$-complex which is homotopy equivalent to $\cC$.   
\end{lemma}
\begin{proof}
Since $f$ is a maximal self-local equivalence, we have from Lemma~\ref{lem:splitcomplex} a chain complex isomorphism 
\begin{align*}
	\varphi \co \im f \oplus \ker f &\rightarrow C \\
				(x, y) &\mapsto (f|_{\im f})^{-1}(x)+y.
\end{align*}
We will show that $\varphi$ satisfies $\varphi \circ \iota' \simeq \iota \circ \varphi$, where $\iota'= \iota_f \oplus \iotaperp$.  Once this is established, it is then easy to verify that $(\im f, \iota_f) \oplus (\ker f, \iotaperp)$ is an $\iota$-complex. 

Let $H \co C \rightarrow C$ be a chain homotopy between $f \circ \iota$ and $\iota \circ f$, i.e., $f \circ \iota + \iota \circ f = H \circ \d + \d \circ H$. Then for $(x,y) \in \im f \oplus \ker f$, we have
\begin{align*}
	\iota \varphi(x,y) + \varphi \iota'(x,y) &= \iota (f|_{\im f})^{-1}(x) + \iota(y) + (f|_{\im f})^{-1} f \iota (f|_{\im f})^{-1}(x) + \iota(y) + (f|_{\im f})^{-1} f \iota(y) \\
		&= (1+ (f|_{\im f})^{-1} f)(f \iota + \iota f) (f|_{\im f})^{-1} (f|_{\im f})^{-1}(x) +  (f|_{\im f})^{-1} (f|_{\im f})^{-1} f (f \iota + \iota f)(y) \\
		&= \d J (x, y) + J \d (x,y),
\end{align*}
where 
\[ J(x,y) = (1+ (f|_{\im f})^{-1} f) H (f|_{\im f})^{-1} (f|_{\im f})^{-1}(x) + (f|_{\im f})^{-1} (f|_{\im f})^{-1} f H(y).\]
Since $\varphi$ is a chain complex isomorphism, this shows that $(\im f, \iota_f) \oplus (\ker f, \iotaperp)$ is an $\iota$-complex and Lemma~\ref{lem:locally-equiv-isom} implies that $\varphi$ induces an equivalence of $\iota$-complexes.  
\end{proof}

We are now able to show that $(\im f, \iota_f)$ carries the local equivalence type of $\cC$.
\begin{lemma}\label{lem:maxlocal}
If $f$ is a maximal self-local equivalence of $\cC =(C, \iota)$, then $ f\co (C,\iota)  \to (\im f, \iota_f)$ and $(f|_{\im f})^{-1} \co (\im f, \iota_f) \to (C,\iota)$ are local equivalences.
\end{lemma}

\begin{proof}
By Lemma~\ref{lem:iotasplit}, we have that $(\im f, \iota_f) \oplus (\ker f, \iotaperp)$ is an $\iota$-complex which is locally equivalent to $(C,\iota)$.  Note that $U^{-1} H_*(\im f) \cong \F[U,U^{-1}]$.  From this, it is straightforward to verify that $(\im f, \iota_f)$ is an $\iota$-complex and the inclusion $(\im f, \iota_f) \to (\im f, \iota_f) \oplus (\ker f, \iotaperp)$ is a local equivalence.   Therefore, the composition of this inclusion with the local equivalence $\varphi$ from Lemma~\ref{lem:iotasplit} is a local equivalence as well.  However, by definition of $\varphi$, this local equivalence is just given by $(f|_{\im f})^{-1}$.  

Since $\varphi$ is a chain complex isomorphism, $\varphi^{-1}$ is also a local equivalence from $(C,\iota)$ to $(\im f, \iota_f) \oplus (\ker f, \iotaperp)$, by the proof of Lemma~\ref{lem:locally-equiv-isom}.  Composing with the projection to $(\im f, \iota_f)$, we obtain a local equivalence from $(C,\iota)$ to $(\im f, \iota_f)$.  This map is exactly $f$.  
\end{proof}

Moreover, the following lemma shows that the homotopy type of $(\im f, \iota_f)$ is independent of the choice of maximal self-local equivalence.

\begin{lemma}\label{lem:maxhomotopy}
If $f$ and $g$ are maximal self-local equivalences of $\cC =(C, \iota)$, then $(\im f, \iota_f)$ and $(\im g, \iota_g)$ are homotopy equivalent as $\iota$-complexes.
\end{lemma}

\begin{proof}
By Lemma \ref{lem:isomim}, $f|_{\im g}\co \im g \rightarrow \im f$ is an isomorphism of chain complexes. We will show that $f|_{\im g} \circ \iota_g$ and $\iota_f \circ f|_{\im g}$ are chain homotopic. In what follows, the domain is $\im g$ and the codomain $\im f$:
\begin{align*}
	f \circ \iota_g &= f \circ g \circ \iota \circ (g|_{\im g})^{-1} \\
		&\simeq f \circ \iota \circ g \circ (g|_{\im g})^{-1} \\
		&= f \circ \iota \\
		&= (f|_{\im f})^{-1} \circ f \circ f \circ \iota \\
		&\simeq (f|_{\im f})^{-1} \circ f \circ \iota \circ f \\
		&= (f|_{\im f})^{-1} \circ f \circ \iota \circ f \circ (f|_{\im f})^{-1} \circ f \\
		&\simeq (f|_{\im f})^{-1} \circ f \circ f \circ \iota \circ (f|_{\im f})^{-1} \circ f \\ 
		&= f \circ \iota \circ (f|_{\im f})^{-1} \circ f \\ 
		&= \iota_f \circ f,
\end{align*}
as desired. Now by Lemma \ref{lem:locally-equiv-isom}, the map $f|_{\im g}$ induces a homotopy equivalence between the $\iota$-complexes $(\im f, \iota_f)$ and $(\im g, \iota_g)$.
\end{proof}

The above lemma tells us that we can use a maximal self-local equivalence to define an invariant of an $\iota$-complex.  Note that the homotopy type of $(\im f, \iota_f)$ as an $\iota$-complex is independent of the choice of $f$ by Lemma \ref{lem:maxhomotopy}. Moreover, the isomorphism type of $\im f$ as a chain complex over $\F[U]$ is independent of the choice of $f$ by Lemma \ref{lem:isomim}.  This leads to the following definition.   
\begin{definition}
Let $\cC$ be an $\iota$-complex and $f$ a maximal self-local equivalence. The \emph{connected complex} $(\cC_\conn, \iota_\conn)$ is $(\im f, \iota_f)$. 
\end{definition}

\begin{proposition}\label{prop:locally-equiv-iota}
Let $\cC$ and $\cC'$ be locally equivalent $\iota$-complexes. Then $\cC_\conn$ and $\cC'_\conn$ are isomorphic as chain complexes; in particular, $\cC_\conn$ is an invariant of the local equivalence type of $\cC$. Furthermore, the homotopy type of $(\cC_\conn, \iota_\conn)$ as an $\iota$-complex is an invariant of the local equivalence type of $\cC$.
\end{proposition}

\begin{proof}
Since $\cC=(C, \iota)$ and $\cC'=(C', \iota')$ are locally equivalent, there exist
\begin{align*}
	F &\co C \rightarrow C' \\
	G &\co C' \rightarrow C
\end{align*}
such that $F \circ \iota \simeq \iota' \circ F$ and $G \circ \iota' \simeq \iota \circ G$, and $F$ and $G$ induce isomorphisms on $U^{-1}H_*$. 

Let $f \co C \rightarrow C$ and $g \co C' \rightarrow C'$ be maximal self-local equivalences. Then $G \circ g \circ F \circ f$ is a self-local equivalence of $C$. Since $\ker (G \circ g \circ F \circ f) \supseteq \ker f$ and $f$ is maximal, we have that $\ker (G \circ g \circ F \circ f) \subseteq \ker f$, i.e., $\ker (G \circ g \circ F \circ f) = \ker f$. Thus, $g \circ F |_{\im f} \co \im f \rightarrow \im g$ is injective. Similarly, $f \circ G|_{\im g} \co \im g \rightarrow \im f$ is injective. Since we have injective grading-preserving $\F[U]$-equivariant chain maps between the finitely generated $\F[U]$ chain complexes $\im f$ and $\im g$, the maps are isomorphisms.

We would like to show that $(\im f, \iota_f)$ and $(\im g, \iota_g)$ are homotopy equivalent $\iota$-complexes. By Lemma \ref{lem:maxlocal}, it follows that $g \circ F \circ (f|_{\im f})^{-1}$ is a composition of local equivalences, and hence induces a local equivalence between $(\im f, \iota_f)$ and $(\im g, \iota_g)$. (The map in the other direction is given by $f \circ G \circ (g|_{\im g})^{-1}$.)  Since this is also a chain complex isomorphism, Lemma~\ref{lem:locally-equiv-isom} implies that we have in fact constructed a homotopy equivalence of $\iota$-complexes.  
%
\end{proof}

With this, we are ready to define the connected homology of an $\iota$-complex, and consequently, the connected Heegaard Floer homology for a spin rational homology sphere.  The following definition will be useful in defining the connected homology.

\begin{definition}
Let $C$ be a finitely generated graded chain complex over $\F[U]$. The \emph{reduced homology} of $C$ is
\[ H_\red(C) =  \ker \big(U^N \co H_*(C) \rightarrow H_*(C)\big)[-1], \]
for $N \gg 0$, i.e., $H_\red(C)$ is the $U$-torsion submodule of $H_*(C)$ with gradings shifted by one.
\end{definition}

\begin{remark}
Note that if $C=CF^-(Y, \s)$, then $H_\red(C) = \HF_\red(Y, \s)$.
\end{remark}

\begin{definition}
Let $\cC$ be an $\iota$-complex. The \emph{connected homology of $\cC$} is 
\[ H_\conn(\cC) = H_\red(\cC_\conn). \]
If $(Y,\s)$ is a spin rational homology sphere, the {\em connected Heegaard Floer homology of $(Y,\s)$} is 
\[
\HF_\conn(Y,\s) = H_\conn(CF^-(Y,\s),\iota).
\]
\end{definition}

The grading shift is included so that $\HF_\conn(Y,\s)$ is graded isomorphic to a direct summand of $\HF_\red(Y,\s)$, viewed as a quotient of $\HF^+(Y,\s)$. 

With the work of this section, we can easily deduce that $\HF_\conn(Y,\s)$ is a spin homology cobordism invariant.

\begin{proof}[Proof of Theorem~\ref{thm:HFconn}]
That the isomorphism class of $\HF_\conn(Y,\s)$ is an invariant of homology cobordism follows directly from Lemma~\ref{lem:homology-cobordism} and Proposition~\ref{prop:locally-equiv-iota}.  
\end{proof}

%% file: connectedproperties.tex
\section{Properties of connected homology}\label{sec:properties}
\subsection{Properties and a numerical invariant of $H_\conn$}  In this section, we study the properties of $H_\conn$ through connections with the involutive correction terms. We first give the relationship between $H_\conn(\cC)$ and $H_\conn(\cC^*)$.

\begin{proposition}\label{prop:conn-duality}
If $H_\conn(\cC)=\bigoplus_i \cT_{a_i}(n_i)$, then $H_\conn(\cC^*)=\bigoplus_i \cT_{-a_i+2n_i-3}(n_i)$.
\end{proposition}

\begin{proof}
It follows from the definition of the chain complex $\cC_\conn$ that $(\cC^*)_\conn \cong (\cC_\conn)^*$. Indeed, if $f$ is a maximal self-local equivalence of $\cC$, then it is straightforward to verify that $f^*$ is a maximal self-local equivalence of $\cC^*$ and $(\im f)^* \cong \im (f^*)$.

The result now follows immediately from duality and our grading conventions.
\end{proof}

We can also generalize Definition \ref{def:omega} to arbitrary $\iota$-complexes, as follows.
  
\begin{definition}
Let $\cC$ be an $\iota$-complex. Define
\[ \omega(\cC) = \min \{ n \mid U^n H_\conn(\cC)=0 \}. \]
\end{definition}

We first relate $\omega$ to the various correction terms of an $\iota$-complex.  
\begin{proposition}\label{prop:formalomegabounds}
Let $\cC$ be an $\iota$-complex. Then
\begin{enumerate}
	\item $\frac{1}{2}(d(\cC)-\dl(\cC)) \leq \omega(\cC)$,
	\item $\frac{1}{2}(\du(\cC)-d(\cC)) \leq \omega(\cC)$.
\end{enumerate}
\end{proposition}

\begin{proof}  We may assume that any maximal self-local equivalence $\cC \rightarrow \cC$ is surjective.  Indeed, we may begin by replacing $\cC$ with its image under some maximal self local equivalence $f \co \cC \rightarrow \cC$; any maximal self-local equivalence of $(\im f, \iota_f)$ must be surjective, as otherwise, one can derive a contradiction to the maximality of $f$.  Thus $H_\conn(\cC) = H_\red(C)$.  Let $n = \frac{1}{2}(d(\cC)-\dl(\cC))$ and $m=\frac{1}{2}(\du(\cC)-d(\cC))$, which are invariants of local equivalence. The existence of the exact triangle 
\[
\begin{tikzcd}[column sep=small]
H_*(C) \arrow{rr}{(1+\iota)_*} & & H_*(C) \arrow{dl} \\
& H_*(\Cone(1+\iota)) \arrow{ul} & 
\end{tikzcd}
\]
of $U$-equivariant maps implies that if $n>0$, the homology $H_*(C)$ must contain $\mathcal T_{d(\cC)-2}(n)$ as a direct summand, and if $m>0$, the homology $H_*(C)$ must contain $\mathcal T_{d(\cC)+2m-3}(m)$ as a direct summand. We conclude that $\HF_{\conn}(\cC)$ is not annihilated by $U^{\max\{n,m\}-1}$, and $\omega(\cC) \geq \max\{n,m\}$.
\end{proof}

\begin{proof}[Proof of Proposition \ref{prop:omegabounds}]
The result follows immediately from the preceding proposition by letting $\cC=(CF^-(Y, \s), \iota)$.
\end{proof}

One advantage of the invariant $\omega$ is that it is well-behaved under connected sums.  
\begin{proposition}\label{prop:omegasum}
For $\iota$-complexes $\cC$ and $\cC'$, we have 
\[ \omega(\cC \otimes \cC') \leq \max \{ \omega(\cC), \omega(\cC')\}. \]
\end{proposition}

\begin{proof} As before, we may assume that any maximal self local equivalence $\cC \rightarrow \cC$ is surjective, so we may assume that $H_\conn(\cC) = H_\red(C)$. Then we have $\omega(\cC) = \min \{n \mid U^nH_{\red}(C)=0\}$. For concreteness, let $H_*(C) = \F[U] \oplus (\oplus_{i=1}^k \mathcal T(n_i))$, with grading information omitted. Then if we order the numbers $n_i$ such that $n_1 \geq n_2 \geq \cdots \geq n_k$, we have that $\omega(\cC) = n_1$. Similarly we may assume that $H_\conn(\cC') = H_\red(C')$, and thus $\omega(\cC') = \min\{n \mid U^nH_\red(\cC')=0\}$. Let $H_*(C') = \F[U]\oplus(\oplus_{j=1}^{\ell} \mathcal T(m_j))$, and order the numbers $m_j$ such that $\omega(\cC') = m_1$.

We now consider $\cC \otimes \cC'$. The connected homology of this complex is (up to grading shift) a submodule of $H_*(C \otimes_{\F[U]} C')$. By the K\"unneth formula for complexes over $\F[U]$ (see for example \cite[Corollary 6.3]{OS3manifolds2}), the $U$-torsion submodule of $H_*(C \otimes_{\F[U]} C')$ is a direct sum of terms of the following forms
\begin{align*}
&\F[U] \otimes \mathcal T(m_j) = \mathcal T(m_j) \text{ for eacn } m_j \\
&\mathcal T(n_i) \otimes \F[U] = \mathcal T(n_i) \text{ for each } n_i \\
&\mathcal T(n_i) \otimes \mathcal T(m_j) = \mathcal T(\min\{n_i,m_j\}) \text{ for each } n_i, m_j \\
&\Tor(\mathcal T(n_i), \mathcal T(m_j)) = \mathcal T(\min\{n_i,m_j\}) \text{ for each } n_i, m_j.
\end{align*}

\noindent We observe that in no case is there a $U$-torsion element in $H_*(C \otimes_{\F[U]} C')$ which is not annihilated by either $U^{n_1}$ or $U^{m_1}$. Therefore there can be no element in $H_\conn(\cC \otimes \cC')$ which is not annhilated by $U^{\max\{n_1,m_1\}}$. It follows that $\omega(\cC\otimes \cC') \leq \max\{\omega(\cC), \omega(\cC')\}$.
\end{proof}

With the above technical results, we are able to prove the following involutive analog of Stoffregen's linear asymptotics of the Manolescu invariants \cite[Theorem 1.3]{Stoffregenconnectedsum} claimed in the introduction, Theorem \ref{thm:asymptotics}.  For notation, we will use $n\cC$ to mean the tensor product of $\cC$ with itself $n$ times.

\begin{theorem}
Let $\cC$ be an $\iota$-complex. Then
\[ \lim_{n\rightarrow \infty} \frac{\dl(n\cC)}{n} =  \lim_{n\rightarrow \infty} \frac{\du(n\cC)}{n} = d(\cC). \]
In particular, if $(Y,\s)$ is a spin rational homology three-sphere, then
\[ \lim_{n\rightarrow \infty} \frac{\dl(\#_n(Y, \s))}{n} =  \lim_{n\rightarrow \infty} \frac{\du(\#_n(Y, \s))}{n} = d(Y,\s). \]
\end{theorem}

\begin{proof}
Since $d(\cC)$ is additive under tensor products, we have that $d(n\cC)=nd(\cC)$. Proposition~\ref{prop:omegasum} implies that $\omega(n\cC) \leq \omega(\cC)$.  Since $\dl(\cC) \leq d(\cC)$ and $d(\cC) \leq \du(\cC)$, it now follows from Proposition~ \ref{prop:formalomegabounds} that
\[ d(\cC) - \frac{2\omega(\cC)}{n} \leq \frac{\dl(n\cC)}{n} \leq \frac{\du(n\cC)}{n} \leq \frac{2\omega(\cC)}{n} +d(\cC). \]
Since $\omega(\cC)$ is independent of $n$, the result follows.  For the claim about the three-manifold invariants, we use $\cC=(CF^-(Y, \s), \iota)$.

\end{proof}

\subsection{$\iota$-complexes with small connected homology}
In this section, we study $\iota$-complexes with small connected homology.  We begin by characterizing when the connected homology is trivial.  

\begin{proposition}\label{prop:connzero}
Let $\cC= (C, \iota)$  be an $\iota$-complex. Then $H_\conn(\cC) = 0$ if and only if $d(\cC)=\dl(\cC)=\du(\cC)$.
\end{proposition}

First, we need a technical structural lemma about $\iota$-complexes.  
\begin{lemma}\label{lem:ddl}
Let $\cC=(C, \iota)$  be an $\iota$-complex. If $d(\cC)=\dl(\cC)$, then there exists a homotopy equivalent complex $\cC' = (C',\iota')$ such that 
\[ C' = \F[U]\langle x \rangle \oplus \bigoplus_i \Big( \F[U]\langle y_i \rangle \oplus \F[U]\langle z_i \rangle \Big), \]
where $\iota'(x) = x$ and there exist positive integers $n_i$ such that 
\begin{align*}
	\d x &=0 \\
	\d y_i&=U^{n_i}z_i \\
	\d z_i &= 0.
\end{align*}
\end{lemma}

\begin{remark}
Note that $\iota'$ above does not necessary split since $x$ could appear in $\iota'(y_i)$ for some $i$.
\end{remark}

\begin{proof}
Since $C$ is a finitely generated free chain complex over $\F[U]$ with homology rank one (as an $\F[U]$-module), we may assume that
\[ C = \F[U]\langle x \rangle \oplus \bigoplus_{i=1}^p \Big( \F[U]\langle y_i \rangle \oplus \F[U]\langle z_i \rangle \Big), \]
where $\d x = 0$ and $\d y_i = U^{n_i}z_i$ for some $n_i \in \Z_{> 0}$.

We now show that if $d(\cC)=\dl(\cC)$, then $\iota \simeq \iota'$ where $\iota'(x) = x$. Indeed, since $\iota$ is a chain map and induces an isomorphism on $U^{-1}H_*(C)$, we have that
\[ \iota(x) = x + \sum_{i \in I} U^{m_i}z_i, \]
for some $I \subseteq \{1, \dots, p\}$ and $m_i \in \Z_{\geq 0}$. Define
\begin{align*}
	\iota'(x)&=x \\
	\iota'(y_i)&=\iota(y_i) \\
	\iota'(z_i)&=\iota(z_i).
\end{align*}
We have that $\iota \simeq \iota'$, and hence $(C,\iota')$ is a homotopy equivalent $\iota$-complex, via a homotopy $H$ defined on basis elements $x, y_i, z_i$ to be
\begin{align*}
	H(x) &= \sum_{i \in I} U^{m_i-n_i}y_i \\
	H(y_i) &=0 \\
	H(z_i) &= 0.
\end{align*}
To see that $H$ is well-defined, we must show that $m_i -n_i \geq 0$ for all $i \in I$. Indeed, $d(\cC)=\dl(\cC)$ implies that $\iota(x)$ is homologous to $x$ and thus the sum $\sum_{i \in I} U^{m_i}z_i$ is contained in $\im \d$.  Since $\d y_i = U^{n_i} z_i$, we have that $m_i \geq n_i$ for all $i \in I$.  This completes the proof of the lemma.
\end{proof}

\begin{proof}[Proof of Proposition \ref{prop:connzero}]
The ``only if'' direction follows immediately from Proposition \ref{prop:formalomegabounds}.

We now prove the ``if'' direction. By Lemma \ref{lem:ddl}, without loss of generality we may assume that there exists a basis $\{x, y_1, \dots, y_p, z_1, \dots, z_p\}$ for $C$ such that $\d y_i = U^{n_i}z_i$ and $\d x = \d z_i =0$ for all $i$, and $\iota(x)=x$ since $d(\cC)=\dl(\cC)$. 

Now consider $\cC^*=(C^*, \iota^*)$, the dual of $\cC$.  We dualize the above basis to obtain a basis 
\[ \{x^*, y^*_1, \dots, y^*_p, z^*_1, \dots, z^*_p\} \]
for $C^*$ with the property that $x^*$ does not appear in $\iota^*(y^*_i)$ or $\iota^*(z^*_i)$ for any $i$ since no $y_i$ nor $z_i$ appears in $\iota(x)$.

We have that $d(\cC)=\du(\cC)$ if and only if $d(\cC^*)=\dl(\cC^*)$ by duality. We now apply the proof of Lemma \ref{lem:ddl} to the basis $\{x^*, y^*_1, \dots, y^*_p, z^*_1, \dots, z^*_p\}$ to obtain $(\iota^*)'$ with the property that $(\iota^*)'(x^*)=x^*$. It follows that $\cC^*=(C^*, (\iota^*)')$ splits as
\[  (\F[U]\langle x^* \rangle, \id) \oplus (A,  (\iota^*)'|_A).\]
where
\[ A=\bigoplus_i \Big( \F[U]\langle y^*_i \rangle \oplus \F[U]\langle z^*_i \rangle \Big). \]
Therefore, $H_\conn(\cC^*) = 0$, so $H_\conn(\cC) = 0$ as well by Proposition~\ref{prop:conn-duality}.
\end{proof}

We are also able to characterize when $\dim_{\F} H_\conn(\cC) = 1$.  

\begin{proposition}\label{prop:HFconn1}
If $\dim_{\F} H_\conn(\cC) = 1$, then either
\begin{enumerate}
	\item the unique element in $H_\conn(\cC)$ is in grading $d(\cC)-1$ and $d(\cC)=\du(\cC)=\dl(\cC)+2$, or
	\item the unique element in $H_\conn(\cC)$ is in grading $d(\cC)$ and $d(\cC)=\dl(\cC)=\du(\cC)-2$,  
\end{enumerate}
\end{proposition}

\begin{proof}
Since $\dim_{\F} H_\conn(\cC) = 1$, we have that $\cC=(C, \iota)$ is locally equivalent to an $\iota$-complex $\cC'=(C', \iota')$ with $H_*(C')=\F[U] \oplus \F$. Moreover, if $\dim_{\F} H_\conn(\cC) = 1$, then by Proposition \ref{prop:connzero}, at least one of $\dl(\cC)$ or $\du(\cC)$ is not equal to $d(\cC)$. Since $d$, $\dl$, and $\du$ are invariants of local equivalence, it follows that at least one of $\dl(\cC')$ or $\du(\cC')$ is not equal to $d(\cC')$. Consider the exact triangle
\[
\begin{tikzcd}[column sep=small]
H_*(C') \arrow{rr}{(1+\iota')_*} & & H_*(C') \arrow{dl} \\
& H_*(\Cone(1+\iota')). \arrow{ul} & 
\end{tikzcd}
\]
If $\dl(\cC') < d(\cC')$, then it follows that the $\F$ summand must be in grading $d(\cC')-2$ and $\dl(\cC')=d(\cC)-2$. Furthermore, it follows from parity of the gradings in $\HF_{\conn}(\cC)$ that $\du(\cC) = d(\cC)$. Similarly, if $\du(\cC') > d(\cC')$, then it follows that the $\F$ summand must be in grading $d(\cC')-1$ and $\du(\cC')=d(\cC)+2$, while $\dl(\cC)=d(\cC)$. Applying the grading shift from the definition of $H_\conn$, we have the result.
\end{proof}

\begin{corollary}\label{cor:(2,3,7)}
If $\dim_{\F} H_\conn(\cC) = 1$, then $\cC$ is locally equivalent to either the involutive complex for $\Sigma(2,3,7)$ or $-\Sigma(2,3,7)$, up to an overall grading shift.
\end{corollary}

Here, following \cite{OSabsgr} and \cite{HM:involutive}, our orientation convention is that $\Sigma(2,3,7) =S^3_{-1}(T_{2,3})$.

\begin{proof}
If the unique element in $H_\conn(\cC)$ is in grading $d(\cC)-1$ and $d(\cC)=\du(\cC)=\dl(\cC)+2$, then by the proof of Proposition \ref{prop:HFconn1}, the $\iota$-complex $\cC=(C, \iota)$ is locally equivalent to an $\iota$-complex $\cC'=(C', \iota')$ with $H_*(C')=\F[U]\langle x \rangle \oplus \F \langle z \rangle$ and $(1+\iota)_*(x)=z$. Then \cite[Proof of Theorem 1.1]{DaiManolescu} shows us that the action of $\iota_*$ on $H_*(C')$ determines (up to homotopy) the underlying action on the chain level on any free chain complex with homology $H_*(C')$. In particular, $C'$ is homotopy equivalent to 
\[ \F[U]\langle a \rangle \oplus \F[U]\langle b \rangle \oplus \F[U]\langle c \rangle, \]
with $\d c = U(a+b)$ and $\iota''(a)=b$, $\iota''(b)=a$, and $\iota''(c)=c$, which is the $\iota$-complex for $\Sigma(2,3,7)$ \cite{HM:involutive}.

If the unique element in $H_\conn(\cC)$ is in grading $d(\cC)$ and $d(\cC)=\dl(\cC)=\du(\cC)-2$, then we repeat the above argument to conclude that $\cC^*=(C^*, \iota^*)$ is locally equivalent to the involutive complex for $\Sigma(2,3,7)$, i.e., $\cC$ is locally equivalent to the involutive complex for $-\Sigma(2,3,7)$.
\end{proof}

%% file: applications.tex
\section{Applications to homology cobordism}\label{sec:applications}
In this section, we give the applications of connected Floer homology to homology cobordism promised in the introduction.  The arguments will rely on a few direct computations of the connected Floer homology of a certain class of manifolds.  These computations are done in Sections~\ref{sec:surgeries} and \ref{sec:computations}.  

We now discuss a filtration on $\Inv_\Q$ which will yield the filtration on the homology cobordism group described in the introduction.  Recall that $\cP$ denotes the set of subsets of $\N$, partially ordered by inclusion. For $P \in \cP$, define
\[ \cF_P^{\Inv} = \{ [\cC] \in \Inv_\Q \mid H_\conn(\cC) \cong \bigoplus_i \F[U]/U^{n_i} \F[U], n_i \in P \}. \]
The above isomorphism need not respect gradings.  We will often be interested in $[N]=\{1, 2, \dots, N\} \in \cP$.  We now prove that $\cP$ induces a filtration on $\Inv_\Q$.

\begin{proposition}\label{prop:I-filtration}
The collection of subsets $\cF_P^{\Inv}$ induces a filtration of $\Inv_\Q$ by $\cP$.  
\end{proposition}
\begin{proof} That $\cF_P^{\Inv}$ is closed under inverses follows from Proposition~\ref{prop:conn-duality}.  The only remaining point is that $\cF_P^{\Inv}$ is closed under connected sum. Let $[\cC], [\cC'] \in \cF_P^{\Inv}$. Let $\cC = (C, \iota)$ be a representative of the local equivalence class with the property that $H_{\conn}(\cC) = H_\red(C)$. We choose an analogous representative for $\cC'=(C', \iota')$.
Then we see that $H_*(C) \cong \F[U] \oplus \bigoplus_i \F[U]/U^{n_i} \F[U]$ for some collection of $n_i \in P$, and similarly for $H_*(C')$. Then the module $\HF_\conn(\cC \otimes \cC')$ must be, up to a grading shift, a summand of the $U$-torsion submodule of $H_*(C \otimes_{\F[U]} C')$.  But the K\"unneth formula implies that the $U$-torsion submodule of $H_*(C \otimes_{\F[U]} C')$ consists of a direct sum of cyclic modules $\F[U]/U^{m_j}$, each of which appears as a summand of $H_*(C)$ or $H_*(C')$, and therefore $m_j \in P$. Consequently, $[\cC \otimes \cC']$ lies in $\cF_P^{\Inv}$.
\end{proof}

We now define the filtration on $\Theta^3_\Z$ described in the introduction.  Recall that $\mathcal{P}$ denotes the powerset of $\mathbb{N}$, and for $P \in \mathcal{P}$, we define
\[ \cF_P = \{ [Y] \mid \HF_\conn(Y) \cong \bigoplus_i \cT_{a_i}(n_i) , \ n_i \in P \}. \]

\filtered*
\begin{proof}
This follows from Proposition~\ref{prop:I-filtration} by applying the map from $\Theta^3_\Z$ to $\Inv_\Q$ which takes $Y$ to $[(CF^-(Y),\iota)]$.  
\end{proof}

The filtration $\cP$ is effective for studying the subgroup of $\Theta^3_\Z$ generated by surgery on knots in $S^3$ of bounded genus.  

\genus*
\begin{proof}
By \cite[Theorem 3]{Gainullin}, we have that for any $n \in \Z$, 
\[ U^{g(K)+\lceil g_4(K)/2 \rceil} \HF_\red(S^3_{1/n}(K)) =0. \]
In particular, suppose $g(K) < N$ and let $\HF_\conn(S^3_{1/n}(K)) \cong \bigoplus_i \F[U]/U^{n_i} \F[U]$. Then for each $i$, we have $n_i < 3N/2$. It follows that the subgroup generated by $\{ S^3_{1/n}(K) \mid g(K) <N, n \in \Z \}$ is contained in $\cF_{[\frac{3N}{2}-1]}$. There exist L-space knots with any value of $V_0$ (e.g. $T_{2,4n+1}$), so by Theorem \ref{thm:-nLspace}, we have that $\cF_{[p]}/\cF_{[p-1]}$ is non-empty for all $p \in \N$, hence $\Theta^3_N$ is a proper subgroup of $\Theta^3_\Z$.
\end{proof}

It is still an open question as to whether every homology sphere is homology cobordant to one obtained by surgery on some knot in $S^3$.  

Note that the above theorem immediately proves that $\Theta^3_\Z$ is infinitely generated.  Using the invariant $\omega$, we easily can reprove Furuta's theorem that $\Theta^3_\Z$ contains a $\Z^\infty$ subgroup.  

\infinitelygenerated*

\begin{proof}  Let $Y_n = S^3_{-1}(T_{2,4n+1})$.  By \cite[Corollary 1.5]{OSalternating}, we have that $V_0(T_{2, 4n+1})=n$. It then follows from Theorem~\ref{thm:connect-sums} that $\omega(k Y_n) = n$ for any integer $k > 0$.  By Proposition~\ref{prop:conn-duality}, the same holds for $k < 0$.  Therefore, by Proposition~\ref{prop:omegasum}, we see that 
\[
\omega(k_1 Y_1 \# \ldots \#  k_n Y_n) \leq \max _{k_i \neq 0} i.
\]
Therefore, we see that no linear combination $k_1 Y_1 \# \ldots k_N Y_N$ with $k_N$ non-zero can be trivial in homology cobordism, since otherwise we would have 
\[
N = \omega({-k_N} Y_N) = \omega (k_1 Y_1 \# \ldots \# k_{N-1} Y_{N-1}) \leq N-1,
\]
where the leftmost equality is by Theorem \ref{thm:connect-sums} (and Proposition \ref{prop:conn-duality} if $k_N>0$). Hence we have reached a contradiction.  
\end{proof}

More generally, Theorems \ref{thm:infiniteorder} and \ref{thm:infiniterankone} below give sufficient conditions for a homology sphere to be infinite order in $\Theta^3_\Z$. The following proposition will be used in the proofs of Theorems \ref{thm:infiniteorder} and \ref{thm:infiniterankone}. 

\begin{proposition} \label{prop:dlowernotd} Let $Y_1$ and $Y_2$ be integer homology spheres such that 
$\dl(Y_1)<d(Y_1)$ and at least one of the following is true of $Y_2$:
	\begin{itemize}
	\item $Y_2$ is $d$-negative, or
	\item $\HFm_r(Y_2)$ is nonzero only in gradings $r$ such that $r \equiv d \pmod{2}$.
	\end{itemize}
Then $\dl(Y_1 \# Y_2) < d(Y_1 \# Y_2)$.
\end{proposition}
Note that, by Proposition \ref{prop:connzero}, this in particular implies that $\dim_{\F}(\HF_\conn(Y_1 \# Y_2)) \geq 1$.
\begin{proof} Either of the conditions on $Y_2$ is sufficient to guarantee that $\du(Y_2) = d(Y_2)$. But then by Proposition \ref{prop:inequalities},
we have
\begin{align*}
\dl(Y_1 \# Y_2) &\leq \dl(Y_1) + \du(Y_2) \\
				&= \dl(Y_1) +d(Y_2) \\
				&< d(Y_1) + d(Y_2) \\
				&= d(Y_1 \# Y_2).
\end{align*}
Here the third step uses the assumption that $\dl(Y_1) < d(Y_1)$. \end{proof}

The following theorems now follow readily.

\infiniteorder*  
\begin{proof}
This follows immediately from Proposition~\ref{prop:dlowernotd}.  
\end{proof}

\infiniterankone*
\begin{proof}
This follows from Corollary~\ref{cor:(2,3,7)} and Proposition~\ref{prop:dlowernotd}.  
\end{proof}

\begin{remark} As noted in the introduction, Proposition \ref{prop:dlowernotd} and therefore Theorems \ref{thm:infiniterankone} and \ref{thm:infiniteorder} also apply in the case that $Y$ is a $\Z_2$-homology sphere, together with its unique spin structure, and we consider either the group $\Theta^3_{\Z_2}$ or $\Theta^3_{\Z_2}/\Theta^3_L$.\end{remark}




%% file: surgeries.tex
\section{Computations for integer surgeries} \label{sec:surgeries}
\subsection{The mapping cone formula and $\iota$ for surgeries}\label{sec:surgeries-1}
In this section, we will study the behavior of involutive Heegaard Floer homology and the connected Floer homology for certain Dehn surgeries.  This will include the claimed computations for $-1$-surgery on $T_{2,4n+1}$ used in the proof of Theorem~\ref{thm:-1surgerytorus}.

We assume that the reader is familiar with the integer surgery mapping cone formula of \cite{OSinteger}. Let $\CFKi(K)$ be the knot Floer complex of $K \subset S^3$, which is freely generated over $\F[U, U^{-1}]$ and $\Z \oplus \Z$-filtered. For $X \subset \Z \oplus \Z$, let $CX$ denote the subset of $\CFKi(K)$ generated over $\F$ by elements with filtration level $(i, j) \in X$. We will be interested in the case that $a \in X$ implies $b \in X$ for all $b < a$ with respect to the product partial order on $\Z \oplus \Z$; in this case, $CX$ will always be a subcomplex of $\CFKi(K)$. We will be particularly interested in
\begin{align*}
A^-_s &= C\{ i \leq 0 \textup{ and } j \leq s\} \\
B^- &= C\{i \leq 0\}.
\end{align*}
The complex $B^-$ is homotopy equivalent to $CF^-(S^3)$. We also have that $C\{j \leq 0\}$ is homotopy equivalent to $CF^-(S^3)$, and, up to a grading shift, $C\{ j\leq s\}$ is homotopy equivalent to $C\{ j \leq 0\}$ (via multiplication by $U^s$), and thus also homotopy equivalent to $C\{ i \leq 0\}$.  Since $H_*(C\{ i \leq 0\}) \cong \F[U]$ and each of these complexes is finitely generated and free over $\F[U]$, this homotopy equivalence is unique up to homotopy.  

We have chain maps
\[ v_s \co A^-_s \rightarrow B^- \qquad \textup{ and } \qquad h_s \co A^-_s \rightarrow B^-, \]
where $v_s$ is inclusion, and $h_s$ is inclusion into $C\{ j \leq s\}$ followed by the chain homotopy equivalence from $C\{ j \leq s\}$ to $C\{ i \leq 0\}$. Let 
\[ V_s = \dim_\F ( \coker v_{s,*}) \qquad \textup{ and } \qquad H_s = \dim_\F ( \coker h_{s,*}).\]
Recall from \cite[Lemma 2.4]{NiWu} that $V_{s+1} \leq V_s$, and that $V_s=0$ for $s \geq g(K)$, where $g(K)$ denotes the Seifert genus of $K$ \cite[Theorem 1.2]{OSgenus}. We also have that $V_s=H_{-s}$ and $H_s=V_s+s$ (see \cite[Lemmas 2.3 and 2.5]{HLZ} combined with \cite[Lemma 2.6]{HomLidmansymplectic}).  We write $\vec{V}$ for the sequence $\{V_s\}^\infty_{s =0}$, which encodes the values of $V_s$ and $H_s$ for all $s \in \Z$.  

We now define an $\F[U]$-module $M(\vec{V})$ which will be used to describe the Heegaard Floer homology of $-1$-surgery on an L-space knot. For each $s \geq 0$ with $V_s \neq 0$, we have two generators $x_s$ and $x'_s$ in grading $-s(s+1)-2$. We have the relations 
\begin{align*}
	U^{V_s}x_s &=U^{V_s}x'_s &\textup{ for } &s\geq 0 \\
	U^{V_s}x_s &=U^{V_s+\frac{s(s+1)}{2}}x_0 &\textup{ for } &s>0.
\end{align*}
See Figure \ref{fig:gradedrootVs} for a depiction of the module $M(\vec{V})$ as a graded root. Let $J_0$ be the $\F[U]$-equivariant involution on $M(\vec{V})$ that interchanges $x_s$ and $x'_s$. It is clear from the definition of $M(\vec{V})$ that this involution is indeed well-defined.

\begin{figure}[htb!]
\begin{tikzpicture}


	\filldraw (-0.75, 0) circle (2pt) node[left] (){$x_0$};
	\filldraw (0.75, 0) circle (2pt) node[right] (){$x'_0$};
	\filldraw (-0.375, -1) circle (2pt);
	\filldraw (0.375, -1) circle (2pt);
	\filldraw (-1, -1) circle (2pt) node[left] (){$x_1$};
	\filldraw (1, -1) circle (2pt) node[right] (){$x'_1$};
	\filldraw (-1, -3) circle (2pt) node[left] (){$x_2$};
	\filldraw (1, -3) circle (2pt) node[right] (){$x'_2$};
	\filldraw (0, -2) circle (2pt);	
	\filldraw (0, -3) circle (2pt);	
	\filldraw (0, -4) circle (2pt);	
	\filldraw (0, -5) circle (2pt);	
	\filldraw (0, -6) circle (2pt);	
	
	\node[] at (0, -6.5) {$\vdots$};

	\draw [very thick] (-0.75, 0) -- (0, -2);
	\draw [very thick] (0.75, 0) -- (0, -2);
	\draw [very thick] (-1, -1) -- (0, -2);
	\draw [very thick] (1, -1) -- (0, -2);	
	\draw [very thick] (-1, -3) -- (0, -4);
	\draw [very thick] (1, -3) -- (0, -4);	
	\draw [very thick] (0, -2) -- (0, -6);	


\end{tikzpicture}
\caption{A graded root for $M(\vec{V})$ where $V_0=2, V_1=V_2=1,$ and $V_s=0$ for $s \geq 3$.}
\label{fig:gradedrootVs}
\end{figure}
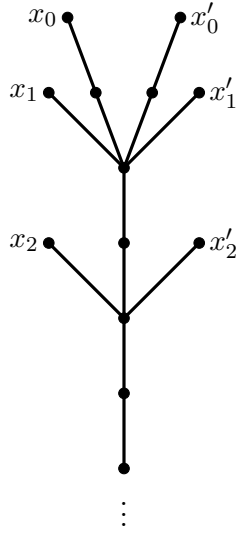

More generally, we define an $\F[U]$-module $M(\vec{V}, n)$ for $n \in \Z_{> 0}$. For each $s \geq 0$ such that $V_{ns}\neq 0$, we have two generators $x_{s}$ and $x'_{s}$ in grading $-ns(s+1)-2$. We have the relations
\begin{align*}
	U^{V_{ns}}x_s &=U^{V_{ns}}x'_s &\textup{ for } &s\geq 0 \\
	U^{V_{ns}}x_s &=U^{V_{ns}+\frac{ns(s+1)}{2}}x_0 &\textup{ for } &s>0.
\end{align*}
As before, let $J_0$ be the $\F[U]$-equivariant involution that interchanges $x_s$ and $x'_s$. Note that $M(\vec{V})=M(\vec{V}, 1)$.  On the other hand, $M(\vec{V},n)$ only depends on the values of $V_k$ for $k \equiv 0 \pmod{n}$.  

\begin{proposition}\label{prop:surgeryinvolution}
Let $K$ be an L-space knot and $n$ a positive integer. Then
\[ \HFm(S^3_{-n}(K), [0]) \cong M(\vec{V}, n)[d(L(n, 1), [0])], \]
and the induced action $\iota_*$ on $\HFm(S^3_{-n}(K))$ coincides with $J_0$ on $M(\vec{V}, n)$.
\end{proposition}

Using Proposition \ref{prop:surgeryinvolution}, we prove the following stronger version of Theorem \ref{thm:-nLspace}.

\begin{theorem}\label{thm:-nLspacefull}
Let $K$ be an L-space knot and $n$ a positive integer. Let $M(\vec{V}, n)[d(L(n, 1), [0])]$ be the $\F[U]$-module described above endowed with the involution $J_0$. Then
\[ \HFI^-(S^3_{-n}(K), [0]) \cong \ker(1+J_0)[-1] \oplus \coker(1+J_0).  \]
Under this isomorphism, the action of $Q$ on $\HFI^-(S^3_{-n}(K), [0])$ is given by the quotient map
\[ \ker(1+J_0) \rightarrow \ker(1+J_0)/\im(1+J_0) \subseteq \coker(1+J_0). \]
The involutive correction terms are
\[ \du(S^3_{-n}(K), [0]) = d(S^3_{-n}(K), [0]) =  -d(L(n,1), [0]) \]
and
\[ \dl(S^3_{-n}(K), [0]) = -2V_0(K) - d(L(n,1), [0]). \]
Finally, $\HF_\conn(S^3_{-n}(K), [0]) = \mathcal{T}_{(-d(L(n, 1), [0])-1)}(V_0(K))$.
\end{theorem}

\begin{proof}

Consider the long exact sequence from \eqref{pic:exact}, which yields a short exact sequence:
\[
0 \to \coker(1 + \iota_*) \to \HFI^-(S^3_{-n}(K),[0]) \to \ker(1+\iota_*)[-1] \to 0,
\] 
where $\iota_*$ is the induced action of $\iota$ on $\HF^-(S^3_{-n}(K),[0])$.  By Proposition~\ref{prop:surgeryinvolution}, $\HF^-(S^3_{-n}(K),[0])$ is all supported in the same grading mod 2, so $\ker(1+\iota_*)[-1]$ and $\coker(1+\iota_*)$ are in different parities of gradings, and the short exact sequence splits.  From this, it follows that $Q$ acts on $\HFI^-$ as claimed.  By the splitting of $\HFI^-$ established, $\dl(S^3_{-n}(K),[0])$ is two more than the maximal degree of a $U$-nontorsion element in $\ker(1 + \iota_*)$, while $\du(S^3_{-n}(K),[0])$ is two more than the maximal degree of a $U$-nontorsion element of $\coker(1 + \iota_*)$.   The involutive correction terms are then easily computed using Proposition~\ref{prop:surgeryinvolution}.  For an alternate and more powerful proof of the involutive Floer homology and correction terms, see the proof of Theorems 1.1 and 1.2 of \cite{DaiManolescu} (combined with Proposition~\ref{prop:surgeryinvolution}).

By Proposition~\ref{prop:surgeryinvolution}, $\HFm(S^3_{-n}(K),[0]) \cong M(\vec{V},n)$ up to a grading shift, where $\iota_*$ on $\HFm(S^3_{-n}(K),[0])$ coincides with the involution $J_0$ on $M(\vec{V},n)$.    Note that the module $M(\vec{V}, n)$ together with $J_0$ has the structure of a symmetric graded root together with its reflection involution, as in \cite[Section 4]{DaiManolescu}.  There exists a $J_0$-equivariant change of basis so that $M(\vec{V},n)$ contains a summand equivariantly isomorphic to $(M(\vec{V}',n),J_0)$, where $\vec{V}' = (V_0,0,0,\ldots)$.  By \cite[Lemma 4.1, Lemma 4.4, and Theorem 6.1]{DaiManolescu}, we see that $\CF^-(S^3_{-n}(K), [0])$ is locally equivalent to a complex with homology isomorphic to $M(\vec{V}',n)$, and thus $\HF_\conn(S^3_{-n}(K),[0])$ is isomorphic to a submodule of $\mathcal{T}_{(-d(L(n, 1), [0])-1)}(V_0(K))$. (The term $-d(L(n,1),[0])$ is from the grading shift between $\HFm(S^3_{-n}(K),[0])$ and $M(\vec{V},n)$, and the $-1$ is from the shift in the definition of $\HF_\conn$.)  Since $d(S^3_{-n}(K), [0])=\dl(S^3_{-n}(K), [0]) + 2V_0$, Proposition~\ref{prop:formalomegabounds} implies that $\HF_\conn(S^3_{-n}(K), [0])$ is exactly $\mathcal{T}_{(-d(L(n, 1), [0])-1)}(V_0(K))$.
\end{proof}



The proof of Proposition \ref{prop:surgeryinvolution} will rely on the integer surgery mapping cone formula of \cite{OSinteger}. In particular, we will use the fact that the maps induced by the 2-handle cobordism from $S^3$ to $S^3_{-n}(K)$ can be computed from the mapping cone, as follows. Throughout, we assume that $n$ is a positive integer.

We will define a map
\[ D_n \co \bigoplus_{s \in \Z} A^-_s \rightarrow \bigoplus_{s \in \Z} B^-_s, \]
where each $B^-_s$ is a copy of $B^- = CF^-(S^3)$.  The map $D_n$ sends $a_s \in A^-_s$ to
\[ D_n(a_s) = v_s(a_s) + h_s(a_s) \]
where
\[ v_s \co A^-_s \rightarrow B^-_s \qquad \textup{ and } \qquad h_s \co A^-_s \rightarrow B^-_{s-n}.\] 
Let $\mathfrak{C} = \operatorname{Cone}(D_n)$, the mapping cone of $D_n$. The absolute grading on the complex $\mathfrak{C}$ is determined by setting the grading of $1 \in H_*(B^-_{-s})$ to be $-2-d(L(n,1),[n-s])$ for $0 \leq s \leq n-1$.

Let $W_{-n}(K)$ be the four-manifold obtained by attaching a 2-handle to $S^3$ along $K \subset S^3 = \d B^4$ with framing $-n$.  Choose a Seifert surface $F$ for $K$ and let $\hat{F}$ denote the capped off surface in $W_{-n}(K)$.    

Fix $N$ a positive integer. Let 
\[ \A^N=\bigoplus_{-N \leq s \leq N} A^-_s  \qquad \textup{ and } \qquad \B^N=\bigoplus_{-N - n \leq s \leq N} B^-_s.\]
Write $\mathfrak{C}^N$ for the subcomplex of the mapping cone $\mathfrak{C}$ given by 
$$
\mathfrak{C}^N = \A^N \oplus \B^N.   
$$
For notational purposes, we denote the differential by $D^N_n$.  Unlike $\mathfrak{C}$, $\mathfrak{C}^N$ is finitely generated over $\F[U]$.  The following allows us to compute the cobordism map associated to $W_{-n}(K)$ in terms of the subcomplex $\mathfrak{C}^N$.   
 
\begin{proposition}\label{prop:mappingconecobordism}
Fix a knot $K \subset S^3$, a positive integer $n > 0$, and a $\spinc$ structure $\mathfrak{t}$ on $W_{-n}(K)$.  There exists $N \gg 0$ such that the following holds.  
\begin{enumerate}
\item \label{it:truncation} $H_*(\mathfrak{C}^N) \cong \HF^-(S^3_{-n}(K))$ as absolutely-graded $\F[U]$-modules.

\item\label{it:cobordisminclusion} If $|\langle c_1(\mathfrak{t}), \hat{F} \rangle|  \leq 2N - n$, then $F^-_{W_{-n}(K), \mathfrak{t}}(1)$ can be computed via the inclusion of $B^-_s$ into $\mathfrak{C}^N$, where $\langle c_1(\mathfrak{t}), \hat{F} \rangle + n = 2s$.   

\item \label{it:cobordismtruncation} If $ |\langle c_1(\mathfrak{t}), \hat{F} \rangle| > 2N - n$, then $F^-_{W_{-n}(K),\mathfrak{t}}(1)$ is the unique non-zero element of $\HF^-(S^3_{-n}(K), \mathfrak{t}|_{S^3_{-n}(K)})$ in degree $\frac{c_1(\mathfrak{t})^2-7}{4}$.  

\end{enumerate}
\end{proposition}
While statements of the form found in Proposition~\ref{prop:mappingconecobordism} are more standard for $\HF^+$, we find that for our use here, the version for $\HF^-$ is more suitable for our computations. We postpone the proof of Proposition \ref{prop:mappingconecobordism} to Section \ref{sec:surgeries-technical}.  Assuming this proposition, we now proceed towards the proof of Proposition~\ref{prop:surgeryinvolution}.

When $K$ is an L-space knot, it follows from \cite[Theorem 4.4]{OSknots} and \cite[Theorem 9.6]{OS3manifolds2} that $H_*(A^-_s) \cong \F[U]$ for all $s$. Furthermore, we will show that if $K$ is an L-space knot, we have the following identification.

\begin{lemma}\label{lem:homologyfirst}
If $K$ is an L-space knot, then
\[ \HF^-(S^3_{-n}(K)) \cong \coker D^N_{n,*}, \]
where 
\[ D^N_{n,*} \co  \bigoplus_{-N \leq s \leq N} H_*(A^-_s) \rightarrow \bigoplus_{-N - n \leq s \leq N} H_*(B^-_s). \]
\end{lemma}

\begin{proof}
We have an exact triangle
\[
\begin{tikzcd}[column sep=small]
H_*(\A^N) \arrow{rr}{D^N_{n,*}} & & H_*(\B^N) \arrow{dl} \\
& H_*(\mathfrak{C}^N) \arrow{ul} & 
\end{tikzcd}
\]
which yields the short exact sequence
\[ 0 \rightarrow \coker D^N_{n,*} \rightarrow H_*(\mathfrak{C}^N) \rightarrow \ker D^N_{n,*} \rightarrow 0 . \]
Since $K$ is an L-space knot, $H_*(A^-_s) \cong \F[U]$ for all $s$. It is straightforward to verify in this case that $v_s$ and $h_s$ are injective on homology for each $s$, which can then be used to prove that $D^N_{n,*}$ is injective. Hence $H_*(\mathfrak{C}^N) \cong \coker D^N_{n,*}$. The result now follows from Proposition \ref{prop:mappingconecobordism}.
\end{proof}

\begin{lemma}\label{lem:cokerD}
Let $n$ be a positive integer. Then
\[ \coker D^N_{n,*} \cong M(\vec{V}, n)[d(L(n, 1), [0])].\]
\end{lemma}

\begin{proof}
We consider the case $n=1$; the argument readily generalizes to arbitrary $n$. Write $D_*$ for $D^N_{n,*}$. For $0 \leq s \leq N$, identify $x_s$ with $1 \in H_*(B^-_{-s-1}) \cong \F[U]$ and $x'_s$ with $1 \in H_*(B^-_{s}) \cong \F[U]$. For $-N \leq s \leq N$, let $y_s$ denote $1 \in H_*(A^-_{s}) \cong \F[U]$.  Note that $U^{V_0}x_0$ and $U^{V_0}x'_0$ are identified in $\coker D_*$ via
\begin{equation}\label{eq:Dy0}
D_*(y_0) = U^{V_0}x_0 + U^{V_0}x'_0. 
\end{equation}
For $1 \leq s \leq N$, using that $V_{-s} = H_s = V_s + s$, we have 
\begin{align}
\label{eq:Dys} &D_*(y_s) = U^{V_s} x'_s + U^{H_s} x'_{s-1} = U^{V_s} x'_s + U^{V_s + s} x'_{s-1} \\
\label{eq:Dy-s} &D_*(y_{-s}) = U^{V_{-s}} x_{s-1} + U^{H_{-s}} x_{s} = U^{V_s+s} x_{s-1} + U^{V_s} x_{s}.
\end{align}
Because the $y_i$'s form a basis for $\bigoplus_{-N\leq s \leq N} H_*(A^-_s)$,
\eqref{eq:Dy0}, \eqref{eq:Dys}, and \eqref{eq:Dy-s} span the image of $D_*$.  

Note that applying \eqref{eq:Dys} recursively yields that for $1 \leq s \leq N$, 
\[U^{V_s + \frac{s(s+1)}{2}} x'_0 + U^{V_s} x'_s \in \im D_*,\] 
and similarly for $x_s$.  From this, it is straightforward to see that the image of $D_*$ is spanned by 
\[
U^{V_0}x_0 + U^{V_0}x'_0, \ \ \ U^{V_s} x_s + U^{V_s + \frac{s(s+1)}{2}} x_0, \ \ \ U^{V_s} x'_s + U^{V_s + \frac{s(s+1)}{2}} x'_0,
\]
It follows that 
\begin{equation}\label{eqn:HFMV} 
\coker D^N_{n,*}  \cong M(\vec{V}, n)[d(L(n, 1), [0])],
\end{equation}
as desired.
\end{proof}

\begin{proof}[Proof of Proposition \ref{prop:surgeryinvolution}]
Let $K$ be an L-space knot. By Lemmas \ref{lem:homologyfirst} and \ref{lem:cokerD}, we have
\[ \HF^-(S^3_{-n}(K), [0]) \cong \coker D^N_{n,*} \cong M(\vec{V}, n)[d(L(n, 1), [0])].\]

We now prove that under this isomorphism, the induced action $\iota_*$ on $\HF^-(S^3_{-n}(K), [0])$ is identified with the involution $J_0$ on $M(\vec{V}, n)$. By \cite[Theorem 3.6]{OS:four} and \cite[Theorem A]{Zemkegraphcobord}, we have that
\begin{equation} \label{eqn:conjugationinv}
F^-_{W_{-n}(K), \t} = \iota_* \circ F^-_{W_{-n}(K), \overline{\t}},
\end{equation}
where we have used the fact that the involution $\iota_*$ on $\HF^-(S^3)$ is the identity. 

Let $N \gg 0$ be as in Proposition~\ref{prop:mappingconecobordism}.  For $0 \leq s \leq N$, let $x'_s$ denote the image of $F^-_{W_{-n}(K), \t}(1)$, where $\langle c_1(\mathfrak{t}), \hat{F} \rangle = 2s - n$.  By Proposition \ref{prop:mappingconecobordism}~\eqref{it:cobordisminclusion}, this is identified with the image of $1 \in H_*(B_s)$ under the inclusion of $B_s$ into $\mathfrak{C}^N$.  Likewise, let $x_s = F^-_{W_{-n}(K), \overline{\t}}(1)$, which is the image of $1 \in H_*(B_{-s-n})$ under the inclusion map to $H_*(\mathfrak{C}^N)$, where $\langle c_1(\overline{\mathfrak{t}}), \hat{F} \rangle = -2s - n$.  By Equation \eqref{eqn:conjugationinv} it follows that $\iota_*$ interchanges $x_s$ and $x'_s$ for $0 \leq s \leq N$.  Since these elements generate $\coker D^N_{n,*}$ over $\F[U]$, we see that $\iota_*$ agrees with $J_0$ under the isomorphism in Equation \eqref{eqn:HFMV}.    
\end{proof}

\begin{proof}[Proof of Theorem~\ref{thm:-ngeneral}]
The proof is nearly identical to the proof of Proposition \ref{prop:surgeryinvolution} and Theorem \ref{thm:-nLspace}. The only difference is that now $\coker D^N_{n,*}$ is a submodule of, rather than isomorphic to, $H_*(\mathfrak{C}^N)$. The generator $U^{V_0}x_0$ is still fixed by $\iota$. The grading of $U^{V_0}x_0$ is $-2V_0-d(L(n, 1), [0])-2$, and the result follows.
\end{proof}

\subsection{The proof of Proposition~\ref{prop:mappingconecobordism}}\label{sec:surgeries-technical}
In this subsection we prove Proposition~\ref{prop:mappingconecobordism}.    

First, we must recall a bit more about the mapping cone formula for $\HF^-$.  It turns out that a direct analogue of the usual mapping cone formula from \cite{OSinteger} does not work for the minus flavor of Heegaard Floer homology.  In \cite{MOlink}, Manolescu and Ozsv\'ath instead establish an analogue with $\F[[U]]$ coefficients. One needs to work with $\F[[U]]$ coefficients becaues the cobordism maps appearing in the mapping cone formula can be non-zero for infinitely many $\spinc$ structures in the minus flavor. We summarize the relevant construction and facts below, before using this to prove Proposition~\ref{prop:mappingconecobordism}.  Throughout, we will write a term in bold to indicate tensoring a finitely-generated module over $\F[U]$ with the power series ring $\F[[U]]$, e.g. $\Ab_s = A^-_s \otimes \F[[U]]$.  Technically, the relevant objects are no longer chain complexes, because they are not a direct sum of their grading homogeneous pieces, but the homological constructions we will use (e.g., grading homogeneous elements, mapping cones, etc.) still make sense.

Fix a non-zero integer $p$.  We will be interested in the cone of 
\[
\psi_p \co \prod_{s \in \Z} \Ab_s \to \prod_{s \in \Z} \Bb_s, \quad (s,a) \mapsto (s,\vb(a)) + (s+p, \hb(a)).
\]    
Note that in this cone, which we denote by $\mathfrak{C}^-(p,K)$, we use direct products instead of direct sums.  (We do not use the bold notation, because this does not arise as the $U$-completion of a finitely-generated module over $\F[U]$.)  As in the uncompleted case, $\Bb_s = \mathbf{CF}^-(S^3)$.  We have the following:
\begin{theorem}{\cite[Theorem 1.1]{MOlink}}\label{thm:MOsurgery}
Fix a non-zero integer $p$ and a knot $K$ in $S^3$.  There is an isomorphism of relatively-graded $\F[U]$-modules 
\begin{equation}\label{eq:completed-cone}
H_*(\mathfrak{C}^-(p,K)) \cong \HFb(S^3_p(K)).
\end{equation}
\end{theorem}
This is a direct analogue of the mapping cone formula for computing $\HF^+$ of surgery, given by Ozsv\'ath-Szab\'o \cite{OSinteger}.  While the statement in Theorem~\ref{thm:MOsurgery} is relatively-graded, we will upgrade this to respect the absolute gradings for negative surgeries in the proof of Proposition~\ref{prop:mappingconecobordism}.    

Recall that we would also like to compute the cobordism map from $\HF^-(S^3)$ to $\HF^-(S^3_p(K))$ associated to the trace of the surgery, $W_p(K)$.  We index the $\spinc$ structures on $W_p(K)$ by
\begin{equation}\label{eq:cobordism-spinc}
\langle c_1(\t_s), [\hat{F}] \rangle + p = 2s, \quad s \in \Z.
\end{equation}
Here, $\hat{F}$ denotes a capped off Seifert surface for $K$ in $W_p(K)$.  Consequently, the above equation implicitly depends on a choice of orientation of $\hat{F}$, hence $K$.  This will not matter for us, since we are only interested in comparing the different $F^-_{W_p(K),\t_s}$ to each other, and determining the isomorphism type of $(\HF^-(S^3_{p}(K)),\iota_*)$, not identifying any of these as canonical elements. 


With completed coefficients, the cobordism map $\Fb_{W_p(K),\t_s}$ can be computed as follows.
\begin{theorem}{\cite[Theorem 14.3]{MOlink}}\label{thm:MOcobordism}
Under an isomorphism between $H_*(\mathfrak{C}^-(p, K))$ and $\HFb(S^3_p(K))$ as in \eqref{eq:completed-cone}, the inclusion from $\Bb_s$ into $\mathfrak{C}^-(p,K)$ induces $\Fb_{W_p(K),\t_s}$ on homology for each $s$.
\end{theorem}
   
In order to establish Proposition~\ref{prop:mappingconecobordism}, we must first translate the above results from $\HFb$ back to $\HF^-$.  We will need the following technical lemma about completions with respect to $U$.  This is well-known, but we include it for completeness.  (See \cite[Section 2]{MOlink} for a similar discussion.)

\begin{lemma}\label{lem:completions}
Let $C$ and $D$ be finitely-generated, free, relatively $\Z$-graded complexes over $\F[U]$, where $U$ has degree $-2$.  Suppose that $\eta \co \mathbf{C} \to \mathbf{D}$ is a grading homogeneous chain map.  Then, there exists a chain map $f \co C \to D$ such that $\mathbf{f}_* = \eta_*$.  Further, $f_*$ is unique.  Finally, $\eta_*$ is an isomorphism of $\F[[U]]$-modules if and only if $f_*$ is an isomorphism of $\F[U]$-modules.     
\end{lemma}
\begin{proof}
First, we recall that $\F[[U]]$ is flat over $\F[U]$, and hence there is a canonical isomorphism $H_*(\mathbf{C}) \cong H_*(C) \otimes \F[[U]]$ and similarly for $D$.  Consequently, by the assumptions on $C$,  
$$
H_*(C) \cong (\F[U])^b \oplus \bigoplus^k_{i=1} \F[U]/(U^{m_i}),
$$
and similarly for $D$.  We then see that 
$$
H_*(C) \otimes \F[[U]] \cong (\F[[U]])^b \oplus \bigoplus^k_{i=1} \F[U]/(U^{m_i})$$
and the result is easily deduced.  
\end{proof}

With this, we are ready to prove the claimed technical proposition.  
\begin{proof}[Proof of Proposition~\ref{prop:mappingconecobordism}]
Fix a knot $K \subset S^3$, a positive integer $n$, and let $N \gg 0$.
\begin{enumerate}
\item \label{it:truncation} By \cite[Lemma 4.4 and Lemma 10.1]{MOlink}, we see that when $n > 0$, the inclusion of $\Cb^N$ into $\mathfrak{C}^-(-n,K)$ is a quasi-isomorphism.  (This is the analogue of the standard truncation for the mapping cone for $\HF^+$ found in \cite[Lemma 4.3]{OSinteger}.)  Combining this with Theorem~\ref{thm:MOsurgery}, we see that $H_*(\Cb^N) \cong \HFb(S^3_{-n}(K))$.  Lemma~\ref{lem:completions} now implies that $H_*(\mathfrak{C}^N) \cong \HF^-(S^3_{-n}(K))$ as relatively-graded $\F[U]$-modules.  We will return to the absolute grading at the end of the proof.    

\item \label{it:cobordisminclusion} Note that if $|s| \leq N$, then the inclusion of $\Bb_s$ into $\mathfrak{C}^-(-n,K)$ factors through $\Cb^N$.  Therefore, if $|\langle c_1(\t_s), [\hat{F}] \rangle| \leq 2N - n$, then $\Fb_{W_{-n}(K),\t_s}$ is computed from the inclusion of $\Bb_s$ into $\Cb^N$.  The result again follows from Lemma~\ref{lem:completions}.  
 
\item \label{it:cobordismtruncation} The last part of the proposition does not need the mapping cone formula.  Since $W_{-n}(K)$ is a negative-definite cobordism, the induced cobordism map must localize to be an isomorphism on $\HF^\infty$, and thus $F^-_{W_{-n}(K),\t}(1)$ is non-zero.  It thus remains to show that this element is unique in its grading.  Since $\HF^\infty(S^3_{-n}(K),\s) \cong \F[U,U^{-1}]$ for any $\s$, there is at most one non-zero element in each sufficiently negative degree of $\HF^-(S^3_{-n}(K),\s)$.  The result now follows, since $|\langle c_1(\t) , [\hat{F}] \rangle| \gg 0$ implies that 
\[
\gr(F^-_{W_{-n}(K),\t}(1)) = \frac{c_1(\t)^2 - 7}{4} \ll 0.
\]  
\end{enumerate}
To complete the proof, it suffices to show that the isomorphism $H_*(\mathfrak{C}^N) \cong \HF^-(S^3_{-n}(K))$ respects the absolute grading.  Note that the absolute grading on $\HF^-(S^3_{-n}(K))$ is determined by that of one element in $\HF^-(S^3_{-n}(K),\s)$ for each $\s$.  By choosing $N > n$, for each $\spinc$ structure $\s$ on $S^3_{-n}(K)$, there exists $s$ so that $|\langle c_1(\t_s), [\hat{F}] \rangle| \leq 2N - n$ and $\t_s \mid_{S^3_{-n}(K)} = \s$.  The definition of the absolute grading on $\mathfrak{C}^N$ is chosen such that the generator of $H_*(B_s) \cong \F[U]$ is precisely in degree $\gr(F^-_{W_{-n}(K),\t_s}(1))$.  The absolute gradings must therefore agree since $F^-_{W_{-n}(K),\t_s}(1) \neq 0$ by the negative-definiteness of $W_{-n}(K)$.
\end{proof}

%% file: HFconncomputations.tex
\section{Computations of $\HF_\conn$}\label{sec:computations}

\subsection{Connected homology of connected sums of surgeries on $L$-space knots} \label{subsec:connectsums} In this section we prove Theorem \ref{thm:connect-sums}, computing the connected Heegaard Floer homology of a connected sum of $-1$-surgeries on L-space knots. We begin by introducing some notation. For $n \in \Z_{>0}$, let $\cC_n$ denote the chain complex appearing in Figure \ref{fig:yncomplex} with involution given by reflection across the centerline of the page. More precisely, $\cC_n$ is generated over $\F[U]$ by $x_1, x_2, y$ such that $\gr(x_1)=\gr(x_2)=-2$, $\gr(y) = -2n-1$, $\del(y) = U^n(x_1+x_2)$, and the involution $\iota$ interchanges $x_1$ and $x_2$ and fixes $y$.

Recall that in Proposition \ref{prop:surgeryinvolution}, we showed that if $K$ is an L-space knot, $\HFm(S^3_{-1}(K)) \cong M(\vec{V}, 1)$, where $\vec{V}=(V_0(K), V_1(K), \ldots)$. We then used the results of \cite{DaiManolescu} to compute $\HFIm(S^3_{-1}(K))$. Going further, Dai and Manolescu showed that if an $\iota$-complex $\cC$ has homology $M(\vec{V},1)$ with involution given by the reflection $J_0$ as described in Section~\ref{sec:surgeries}, then $\cC$ is locally equivalent to $\cC_{V_0}$ \cite[Theorem 6.1]{DaiManolescu}. Therefore, if $K_1, \cdots, K_m$ are L-space knots, in order to compute $\HF_{\conn}(\#_{i=1}^{m} S^3_{-1}(K_i))$, it suffices to compute the connected homology of $\cC_{V_0(K_1)} \otimes \cdots \otimes \cC_{V_0(K_m)}$.  Recall that if $K_i$ is the unknot, then $S^3_{-1}(K_i)$ is trivial in homology cobordism, so it suffices to assume that $K_i$ is a non-trivial L-space knot, or equivalently, $V_i > 0$ for all $i$.  Theorem \ref{thm:connect-sums} is thus a corollary of the following.

\begin{figure}
\begin{center}
\includegraphics{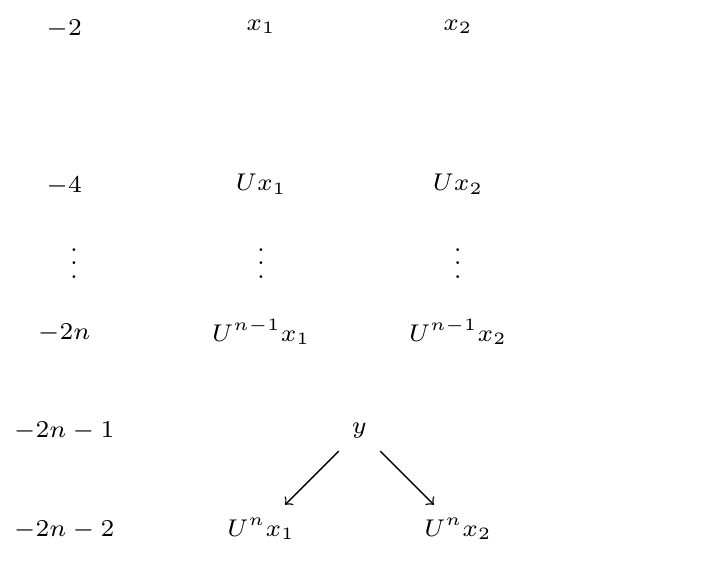}
\end{center}
\caption{The highest gradings in the complex $\cC_n$. The complex continues with further $U$-powers of the three generators. The involution is reflection across the centerline of the page.}
\label{fig:yncomplex}
\end{figure}

\begin{proposition} \label{prop:connected-comp} Let $n_1\geq n_2 \geq \cdots \geq n_m > 0$. Let $a_i = \sum_{j=1}^{i} n_j$, with $a_0=0$. Then $H_{\conn}(\cC_{n_1} \otimes \cC_{n_2} \otimes \cdots \otimes \cC_{n_m})$ is given by
\[
\bigoplus_{i=1}^{m} \mathcal{T}_{i-2-2a_{i-1}}(n_i).
\]
\end{proposition}

As a warm-up, we present the argument for $m=1$. (The argument given in the proof of Theorem \ref{thm:-nLspacefull} could also be used to establish this first lemma more succinctly, but would not help us as much in the general case.)

\begin{lemma} \label{lemma:case-one} $H_{\conn}(\cC_n) = \mathcal{T}_{-1}(n)$.
\end{lemma}

\begin{proof}
As above, let $\cC_n = (C_n, \iota)$ be generated by $x_1$, $x_2$, and $y$ such that $\del(y) = U^n(x_1+x_2)$, $\gr(x_1)=\gr(x_2)=-2$, $\gr(y) = -2n-1$, and $\iota$ fixes $y$ and interchanges $x_1$ and $x_2$. Observe that $x_1$ and $x_2$ are both generators of $U^{-1}H^*(C_n)$, in which $[x_1]=[x_2]$.  Since $H_{\red}(\cC_n) = \mathcal{T}_{(-1)}(n)$, it suffices to prove that any self-local equivalence is surjective.   

Suppose that $f \co C_n \rightarrow C_n $ is a self-local equivalence. Let $G \co C_n \rightarrow C_n$ be a chain homotopy such that $\del G + G \del = f \circ \iota + \iota \circ f$.  

Now, consider $x_1$. Because $f$ preserves the homological grading, $f(x_1) = \lambda_1 x_1 + \lambda_2 x_2$ for $\lambda_1, \lambda_2 \in \{0,1\}$. However, notice that if $f(x_1)=0$, then in the map $f_* \co U^{-1} H_*(C_n) \rightarrow U^{-1} H_*(C_n)$, we have $f_*([x_1])=0$, so $f_*$ is not an isomorphism. Similarly if $f(x_1) = x_1+x_2$, then $f_*([x_1]) = [x_1]+[x_2]=0$ as a map $U^{-1} H_*(C_n) \rightarrow U^{-1} H_*(C_n)$. So in fact either $f(x_1)=x_1$ or $f(x_1)=x_2$. 

Moreover, $G(x_1)$ must have grading $-1$, implying that $G(x_1)=0$. Hence 
\[
f(\iota(x_1))+ \iota(f(x_1)) = \del G(x_1) + G(\del(x_1)) = 0.
\]
This implies that $f(x_2)=\iota(f(x_1))$. So $f(x_1 + x_2) = x_1 + x_2$. Now, consider $f(y)$. Again for grading reasons, either $f(y) = y$ or $f(y) = 0$. But $\del(f(y)) = f(\del y) = f(U^n(x_1+x_2)) = U^nf(x_1+x_2)=U^n(x_1+x_2)$. So $f(y)$ cannot be zero, and we have that $f(y)= y$. We see $f$ is a surjection and $H_{\conn}(\cC_n) = H_\red(C_n) = \mathcal T_{-1}(n)$.
\end{proof}

\begin{figure}
\begin{center}
\includegraphics{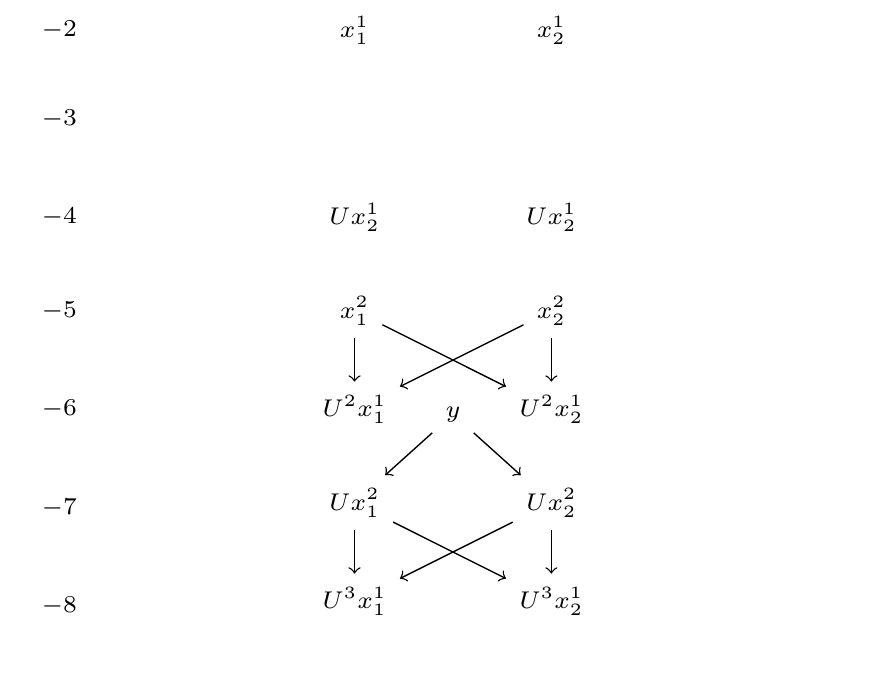}
\end{center}
\caption{The highest gradings of the complex $\cC_{2,1}$. The complex continues with further $U$-powers of the five generators. The involution is given by reflection across the centerline of the page.}
\label{fig:example}
\end{figure}

\begin{lemma} \label{lemma:reduced} Let $n_1\geq n_2\geq \cdots \geq n_m > 0$. Then $\cC_{n_1} \otimes \cdots \otimes \cC_{n_m}$ is locally equivalent to the complex $\cC_{n_1, \cdots, n_m}$ generated by elements $x_i^j$ such that $i \in \{1,2\}, 1\leq j \leq m$ and an element $y$ such that
\begin{align*}
\gr(x_i^1)&=-2 \\
\gr(x_i^j)&= -2a_{j-1} + (j-3) \textup{ for } 2 \leq j \leq m \\
\gr(y) &= -2a_m + m-2
\end{align*}
with differentials
\begin{align*}
\del(x_i^1) &= 0 \\
\del (x_i^j) &= U^{n_{j - 1}}(x_1^{j-1} + x_2^{j-1}) \textup{ for } 2 \leq j \leq m \\
\del (y) &= U^{n_m}( x_1^m + x_2^m),
\end{align*}%
and involution $\iota$ interchanging $x_1^j$ and $x_2^j$ and fixing $y$.
\end{lemma}

As an example, Figure \ref{fig:example} shows the complex $\cC_{2,1}$.

\begin{proof} We proceed by induction on $m$. The base case $m=1$ is true by definition. Suppose that we know the theorem up to $m-1$ inputs. Then given integers $n_1 \geq n_2 \geq \cdots \geq n_m$, we know that $\cC_{n_2} \otimes \cdots \otimes \cC_{n_m}$ is locally equivalent to $\dD = \cC_{n_2, \cdots, n_m}$. So it suffices to show that $\cC_{n_1} \otimes \dD$ is locally equivalent to $\cC_{n_1, \cdots, n_m}$.

For this argument, let $\cC_{n_1}$ consist of the chain complex generated by $s_1, s_2, t$ such that $\gr(s_1)=\gr(s_2)=-2$, $\gr(t) = -2n_1-1$, and $\del(t) = U^{n_1}(s_1 + s_2)$, with involution $\iota$ exchanging $s_1$ and $s_2$ and fixing $t$. (That is, we change notation to avoid repeating $x$ and $y$.) The complex $\dD$ has elements $c$, $b_i^j$ such that $\gr(b_i^2) = -2$, $\gr(b_i^j) = -2(n_2+\cdots+n_{j-1}) + (j-4)$ for $3 \leq j \leq m$, and $\gr(c) = -2(n_2 + \cdots + n_m) + (m-3)$, with differentials
\begin{align*}
&\del(c) = U^{n_m}(b_1^m + b_2^m) \\
&\del(b_i^j) = U^{n_{j-1}}(b_1^{j-1} + b_2^{j-1}) \text{ for } 3 \leq j \leq m \\
& \del (b_i^2) = 0.
\end{align*}
It is important to note that $b_i^j$ corresponds to $x_i^{j-1}$ in the notation of the lemma when applied to $\cC_{n_2,\ldots,n_m}$.  We have chosen this to simplify the notation in the following computations.  The underlying chain complex of $\cC_{n_1} \otimes \dD$ has generators $s_k \otimes b_i^j$, $t \otimes b_i^j$, $s_k \otimes c$, and $t\otimes c$, where throughout $i,k \in \{1,2\}$ and $2 \leq j \leq m$. The gradings of these generators are
\begin{align*}
&\gr(s_k \otimes b_i^j) = -2(n_2+ \cdots + n_{j-1}) + (j-4) \\
&\gr(t \otimes b_i^j) = -2(n_1 + \cdots + n_{j-1}) + j-3 = -2a_{j-1} + (j-3) \\
&\gr(s_k \otimes c) = -2(n_2 + \cdots + n_m) + m-3 \\
&\gr(t \otimes c) = -2(n_1 + \cdots + n_m) + m -2 = -2a_m + m-2,
\end{align*}
where we are using that the tensor product of $\iota$-complexes has a grading shift by 2 in the definition.  The differentials are
\begin{align*}
&\del(s_k \otimes b_i^2) = 0 \\
&\del(t \otimes b_i^2) = U^{n_1}(s_1 \otimes b_i^2 + s_2 \otimes b_i^2)\\
&\del(s_k \otimes b_i^j) = U^{n_{j-1}}(s_k \otimes b_1^{j-1} + s_k \otimes b_2^{j-1}) \text{ for }3 \leq j \leq m\\
&\del (t \otimes b_i^j) = U^{n_{j-1}}(t \otimes b_1^{j-1} + t\otimes b_2^{j-1})+U^{n_1}(s_1 \otimes b_i^j + s_2 \otimes b_i^j) \text{ for } 3 \leq j \leq m  \\
&\del (s_k \otimes c) = U^{n_m}(s_k \otimes b_1^m + s_k \otimes b_2^m) \\
&\del(t \otimes c) = U^{n_m}(t \otimes b_1^m + t \otimes b_2^m) + U^{n_1}(s_1 \otimes c + s_2 \otimes c).
\end{align*}

We now perform a $U$-equivariant and $\iota$-equivariant change of basis. Let 
\begin{align*}
&\wt{y} = t \otimes c \\
&\wt{x}^m_i = t \otimes b_i^m + U^{n_1-n_m}(s_i \otimes c) \\
&\wt{x}^j_i = t \otimes b_i^j + U^{n_1-n_j} (s_i \otimes b^{j+1}_{i+1}) \text{ for } 2\leq j \leq m-1 \\
&\wt{x}^1_i = s_i \otimes b_{i+1}^2 \\
&\wt{s}_k = s_k \otimes b_1^2 + s_k \otimes b_2^2.
\end{align*}
Here $i+1$ is to be taken modulo two; that is, this expression denotes a change of index between $1$ and $2$. Then the elements $\wt{x}^j_i$ for $1 \leq j \leq m$, $\wt{y}$, $\wt{s}_i$, $s_i \otimes c$, and $s_k \otimes b_i^j$ for $3 \leq j \leq m$ generate $\cC_{n_1} \otimes \dD$, with differentials given by
\begin{align*}
&\del(\wt{y}) = U^{n_m}(\wt{x}^m_1 + \wt{x}^m_2) \\
&\del(\wt{x}^j_i) = U^{n_{j-1}}(\wt{x}^{j-1}_1 + \wt{x}^{j-1}_2)  \text{ for } 2 \leq j \leq m \\
&\del(\wt{x}^1_i) = 0 \\
&\del(s_k \otimes c) = U^{n_m}(s_k \otimes b_1^m + s_k \otimes b_2^m) \\
&\del(s_k \otimes b_i^j) = U^{n_{j-1}}(s_k \otimes b_1^{j-1} + s_k \otimes b_2^{j-1}) \text{ for } 3 \leq j \leq m \\
&\del(\wt{s}_k) = 0.
\end{align*}

This complex has a self-local equivalence onto the subcomplex generated by $\wt{y}$ and the elements $\wt{x}^j_i$ for $1 \leq j \leq m$, via sending the remaining generators to zero. Indeed, this subcomplex is exactly $\cC_{n_1, \cdots, n_m}$.
\end{proof}

\begin{proof}[Proof of Proposition \ref{prop:connected-comp}] By Lemma \ref{lemma:reduced}, it suffices to compute the connected homology of the $\iota$-complex $\cC = \cC_{n_1, \cdots, n_m}$. Let $f \co C \rightarrow C$ be a self-local equivalence, and let $G \co C \rightarrow C$ be a chain homotopy such that $f \circ \iota + \iota \circ f = G\del + \del G$.  We will show that $f$ is surjective.  

First we consider $x_1^1$ and $x_2^1$. By the same logic as in Lemma \ref{lemma:case-one}, we see that $f(x_1^1)$ and $f(x_2^1)$ are $x_1^1$ and $x_2^1$ in some order. 

Now, suppose that $f$ is not surjective. First, suppose that some $x_i^j$ is not contained in the image of $f$, and pick a minimal $j>1$ such that $x_i^j$ is not in $\im(f)$ for some $i$. Indeed, since we can postcompose with $\iota$ to get a new self-local equivalence, without loss of generality $x_1^j$ is not in the image of $f$. This implies that either $x_2^j$ is not in $\im(f)$ or $x_1^j + x_2^j$ is not in $\im(f)$, since if they both were then $x_1^j$ would be as well. 

First consider the case that $x_2^j \notin \im(f)$. Then we claim there is no element $z \in \im(f)$ such that $\del(z) = U^{n_{j-1}}(x_1^{j-1}+x_2^{j-1})$. For if such a $z$ existed, then $z$ would lie in grading $-2a_{j-1} + (j-3)$, and therefore could be written uniquely as a linear combination of $x_1^j$, $x_2^j$, and elements $U^{\ell} x_i^{k}$ such that $k<j$ and $-2a_{k-1} + (k-3) - 2\ell = -2a_{j-1} + (j-3)$. Indeed, since $\del(z)= U^{n_{j-1}}(x_1^{j-1}+x_2^{j-1})$, this linear combination would contain exactly one of $x_1^j$ and $x_2^j$. But since $k<j$, each element $U^{\ell}x_i^k$ appearing in this linear combination is in $\im(f)$, so this implies that either $x_1^j$ or $x_2^j$ is in $\im(f)$. Hence there is no such $z$. This implies that $[U^{n_{j-1}}(x_1^{j-1}+x_2^{j-1})]$ is a nontrivial element in $H_*(\im(f))$ in grading $-2a_{j-1} + (j-4)$. If $j$ is odd, this implies $H_*(\im(f))$ in grading $-2a_{j-1} + (j-4)$ is at least one-dimensional, and if $j$ is even, this implies that $H_*(\im(f))$ is at least two-dimensional, since it contains both $[U^{n_{j-1}}(x_1^{j-1}+x_2^{j-1})]$ and $[U^{a_{j-1}-\frac{j}{2}+1}x_1^1]$, and there is no element in $C$ with boundary 
$U^{n_{j-1}}(x_1^{j-1}+x_2^{j-1})+ U^{{a_{j-1}-\frac{j}{2}+1}}x_1^1$. But if $j$ is odd, then $H_*(C)$ has no nontrivial element in grading $-2a_{j-1} + (j-4)$, and if $j$ is even, $H_*(C)$ is one-dimensional generated by $[U^{{a_{j-1}-\frac{j}{2}+1}}x_1^1]$. Since in both cases the map $f_* \co H_*(C) \rightarrow H_*(\im(f))$ is a surjection, in either case we have a contradiction.

Now consider the case that $x_2^j \in \im(f)$, but $x_1^j + x_2^j \notin \im(f)$. Choose any $w$ such that $f(w) = x_2^j$. Then we note that $f(\iota(w)) + \iota (f(w)) = \del G(w) + G\del(w)$, implying that $f(\iota(w)) = x^j_1 + \del G(w) + G\del(w)$. Now each of $\del G(w)$ and $G\del(w)$ is an element in grading $-2a_{j-1} + (j-3)$, and therefore can be written uniquely as a linear combination of $x_1^j$, $x_2^j$, and elements $U^{\ell} x_i^{k}$ such that $k<j$ and $-2a_{k-1} + (k-3) - 2\ell = -2a_{j-1} + (j-3)$. Note that $x_i^j$ cannot appear in $\del G(w)$, because the image of $\del$ is contained in $U \cdot C_{n_1,\ldots,n_m}$, because we have assumed that $n_m > 0$. Similarly, $x_i^j$ cannot appear in $G(\del w)$ because of the $U$-equivariance of $G$. Therefore $f(\iota(w)) = x^j_1 + \del G(w) + G\del(w)$, where $\del G(w)$ and $G\del(w)$ can be written as linear combinations of elements $U^{\ell}x_i^k$ for $k<j$. But $x_i^k \in \im(f)$ for $k<j$, so we see that $x^j_1$ is also in $\im(f)$.

Now suppose every $x_i^j$ is in the image of $f$, but $y$ is not. Then by a similar argument to the first case, there is no element $z \in \im(f)$ such that $\del(z) = U^{n_m}(x_1^m + x_2^m)$. Thus $[U^{n_m}(x_1^m + x_2^m)]$ is a nontrivial  element in $H_*(\im(f))$ in grading $-2a_m +m-3$, which as before implies a contradiction. So $y \in \im(f)$, and in fact $f$ is surjective. 

Since $f$ is surjective, we conclude that $H_{\conn}(\cC_{n_1,\cdots,n_m}) = H_\red(C_{n_1,\cdots,n_m})$. The result follows. 
\end{proof}

\begin{proof} [Proof of Theorem \ref{thm:connect-sums}] Let $K_1, \cdots, K_m$ be L-space knots. By Proposition \ref{prop:surgeryinvolution}, $\HFm(S^3_{-1}(K_j)) \cong M(\vec{V_j}, 1)$, where $\vec{V_j}=(V_0(K_j), V_1(K_j), \cdots)$, and $\iota_*$ is given by reflection on the symmetric graded root. But by \cite[Theorem 6.1]{DaiManolescu}, this implies that the local equivalence class of $(\CFm(S^3_{-1}(K_j)),\iota)$ is represented by $\cC_{V_0(K_j)}$. The theorem then follows directly from Proposition \ref{prop:connected-comp}. \end{proof}

\subsection{Connected homology of graded roots}\label{subsec:gradedroots}

In this section, we give a computation of the connected homology of any $\iota$-complex $(\cC, \iota)$ whose homology consists of a symmetric graded root with induced involution given by its natural symmetry. (This includes $(\CFm(Y), \iota)$ for $Y$ a Seifert fibred space.) Our computation follows fairly quickly from Dai and Manolescu's computation of the local equivalence classes of such $\iota$-complexes \cite{DaiManolescu}. This will prove Theorem \ref{thm:gradedroots}.

First, let us to recall some notation for graded roots. With the conventions appropriate to working with the minus variant of Heegaard Floer theory, recall that a graded root $M$ consists of an infinite tree together with a grading function $\chi \co\mathrm{Vert}(M) \rightarrow \Q$ such that

\begin{itemize}
\item $\chi(u)-\chi(v) = \pm 1$ for any edge $(u,v)$,
\item $\chi(u) \leq \max\{\chi(v),\chi(w)\}$ for any edges $(u,v)$ and $(u,w)$ with $v\neq w$,
\item $\chi$ is bounded above,
\item $\chi^{-1}(k)$ is finite for any $k \in \mathbb Q$,
\item $\#\chi^{-1}(k)=1$ for $k \ll 0$.

\end{itemize}

To every graded root one can associate an $\F[U]$-module $\mathbb H^-(M)$ with one generator for every vertex $v$ in $M$, and we let $U\cdot v = w$ if $(v,w)$ is an edge and $\chi(v)-\chi(w)=1$. (Notice that this means it is simple to reconstruct the graded root from $\mathbb H^-(M)$ together with a preferred set of generators.) Because $\deg(U)=-2$, one typically doubles the relative grading when focusing on this module; from here on we shall do this without comment.

A \emph{symmetric} graded root is a graded root $M$ together with a grading-preserving involution $J_0\co \mathrm{Vert}(M) \rightarrow \mathrm{Vert}(M)$ such that
\begin{enumerate}
	\item for each $k \in \Q$, the involution $J_0$ fixes at most one vertex in $\chi^{-1}(k)$,
	\item $(v,w)$ is an edge if and only if $(J_0v,J_0w)$ is.
\end{enumerate}

A \emph{monotone} graded root is a graded root constructed as follows. Let $h_1, \cdots, h_n$ and $r_1, \cdots, r_n$ be two sequences of rational numbers, all differing from each other by even integers, such that 
\begin{enumerate}
\item $h_1 > h_2 > \cdots > h_n$
\item $r_1< r_2 < \cdots < r_n$
\item $h_n \geq r_n.$
\end{enumerate}
It is most convenient to describe the associated graded root $M=M(h_1,r_1;h_2,r_2;\dots;h_n,r_n)$ by describing the associated module $\mathbb H^-(M)$. If $h_n>r_n$, the monotone graded root is the  tree underlying the $\F[U]$-module with generators $v_i$ and $J_0v_i$ in grading $h_i$ for each $1\leq i\leq n$, with the relationship $U^{\frac{h_i-r_i}{2}}v_i = U^{\frac{h_i-r_i}{2}}J_0v_i = U^{\frac{h_n-r_i}{2}}v_n$. If $h_n=r_n$ then the montone graded root is the  tree underlying the $\F[U]$-module with generators $v_i$ and $J_0v_i$ in grading $h_i$ for each $1\leq i < n$ and a generator $v_n$ in grading $h_n$, with the relationship $U^{\frac{h_i-r_i}{2}}v_i = U^{\frac{h_i-r_i}{2}}J_0v_i = U^{\frac{h_n-r_i}{2}}v_n$. See Figure \ref{fig:monotone} for two examples of monotone graded roots.

\begin{figure}
\begin{center}
\includegraphics{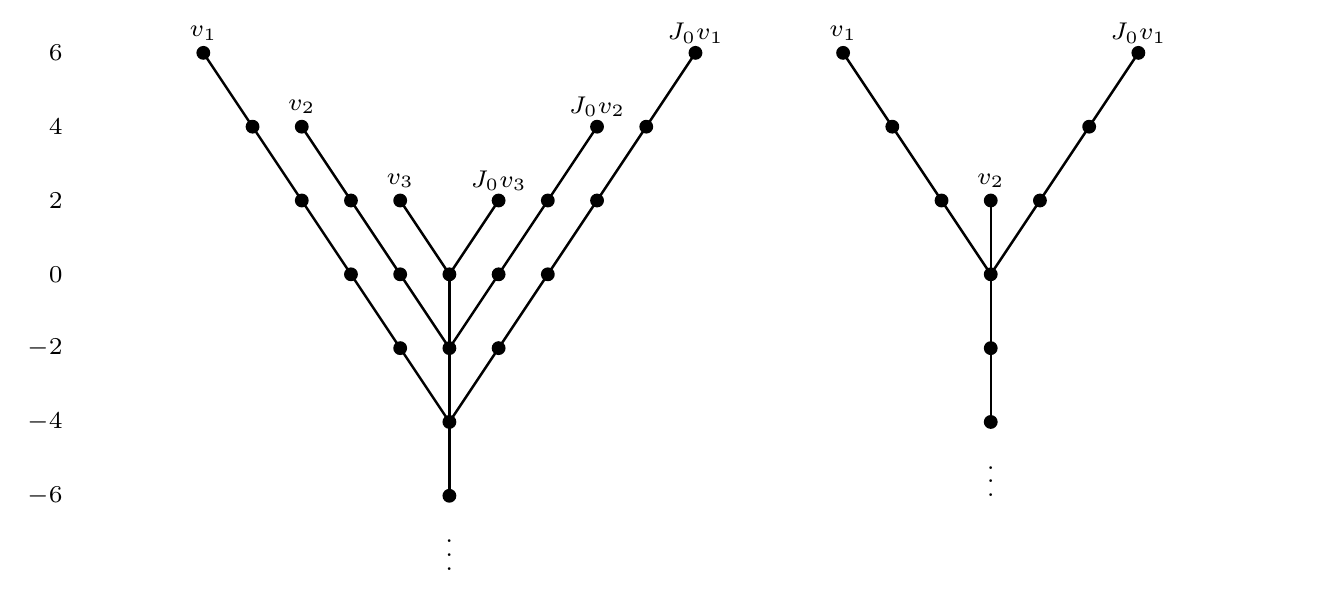}
\end{center}
\caption{The monotone graded roots $M(6,-4;4,-2;2,0)$ (left) and $M(6,0;2,2)$ (right). The numbers to the left denote degree of the elements at the corresponding height.}
\label{fig:monotone}
\end{figure}

Dai and Manolescu associate to any symmetric graded root $M$ a monotone graded subroot as follows \cite[Page 22]{DaiManolescu}. Any symmetric graded root has an infinite downward stem fixed by the involution $J_0$. Given a vertex $v$ of the graded root, one lets $\gamma(v)$ be the infinite path downward through the stem from $v$. The \emph{base} $b(v)$ of this path is the degree of the first place where this path intersects the stem. The collection of all vertices $v$ with the same base is a \emph{cluster} $C_b$. Within every cluster with more than one vertex (called nontrivial), we select a pair of vertices interchanged by $J_0$ with maximal grading, called the \emph{tips} of the cluster.

The algorithm proceeds by constructing a special subset $S$ of the vertices of $M$ as follows: One lets $r$ be the maximal degree of a $J_0$-invariant vertex $v$ in $M$. If the cluster $C_r$ is trivial, we add $v$ to $S$; otherwise we add the tips of $C_r$ to $S$. Now we let $b$ be the greatest number strictly less than $r$ for which $C_b$ is nontrivial. If the tips of $C_b$ have grading greater than the degree of all vertices in $S$, we add them to $S$; otherwise we do not. We iterate this process until there are no further numbers $b$ for which $C_b$ is nontrivial. The monotone graded subroot $M'$ of $M$ is the smallest subroot containing all the vertices in $S$; identifying $M$ with $\mathbb H^-(M)$, it is the span of the generators associated to the elements of $S$ in $\mathbb H^-(M)$. See Figure \ref{fig:subroot} for an example.

\begin{figure}
\begin{center}
\includegraphics{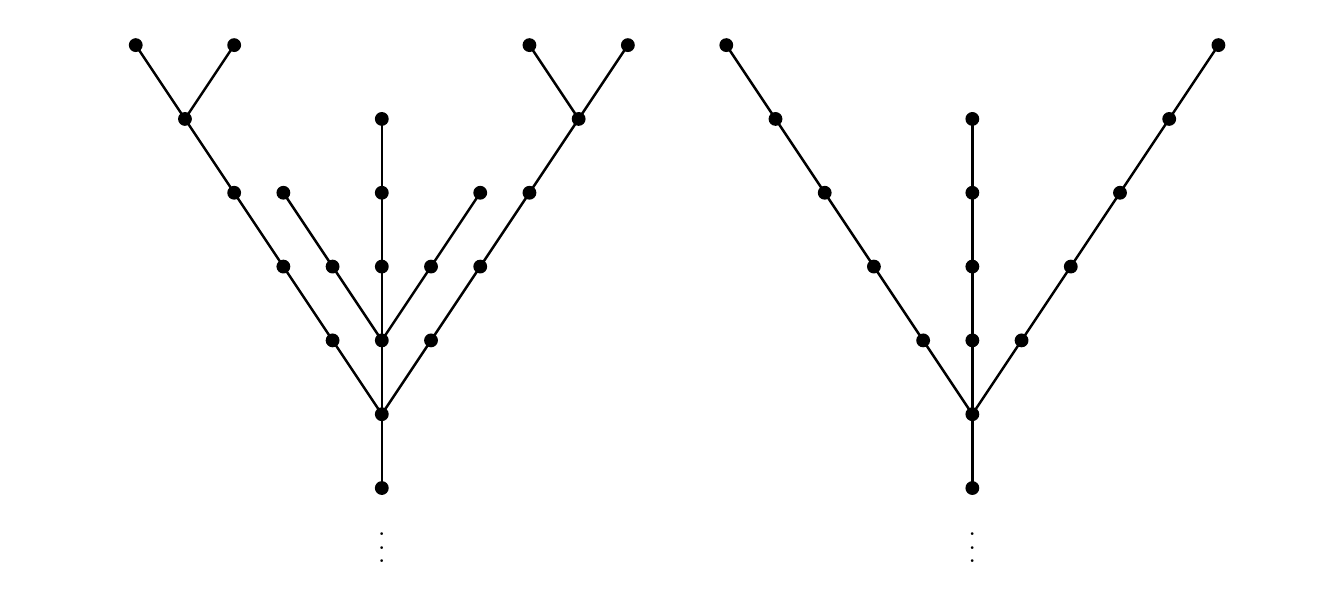}
\end{center}
\caption{A graded root and its monotone subroot.}
\label{fig:subroot}
\end{figure}

In \cite{DaiManolescu}, the authors associate to a symmetric graded root $M$ together with its natural symmetry an $\iota$-complex $(C_*(M), \iota)$ whose homology $H_*(M)$ is the module $\mathbb H^-(M)$ with its involution $J_0$. They show the following.

\begin{theorem} \cite[Corollary 4.6, Theorem 6.1]{DaiManolescu} Let $\cC$ be any $\iota$-complex whose homology is the $\F[U]$-module $H_*(M)$ determined by the graded root $M$ with induced involution given by the natural symmetry $J_0$. Then $\cC$ is locally equivalent to the chain complex $(C_*(M), \iota)$. Moreover, let $M'$ be the monotone graded subroot of $M$ constructed as above. Then $(C_*(M), \iota)$ is locally equivalent to $(C_*(M'), \iota)$.
\end{theorem}

From this, we can compute the following.

\begin{proposition} Let $M = M(h_1,r_1;h_2,r_2;\cdots;h_n,r_n)$ be a monotone graded root, with associated complex $(C_*(M), J_0)$ and homology $H_*(M)$. Then the involutive connected homology of $C_*(M)$ is the $U$-torsion submodule of $M$, shifted upward in degree by 1. \end{proposition}

\begin{proof} It suffices to show that any self-local equivalence $f \co C_*(M) \rightarrow C_*(M)$ necessarily induces a surjection $f_* \co H_*(M) \rightarrow H_*(M)$. In particular, since $f_*$ is $U$-equivariant it suffices to show that each of the generators $v_1,J_0v_1,\dots, v_n, J_0v_n$ are in the image of $f_*$ (here $v_n$ and $J_0v_n$ may be equal). We proceed by induction on $n$, essentially mimicking the proof of \cite[Theorem 6.2]{DaiManolescu}. 

First, observe that if $v_1 = J_0v_1$, then $H_*(M) \cong \F[U]$, and the statement is trivially true. So we assume that $v_1 \neq J_0v_1$. Consider $f_*(v_1)$. This must be a linear combination of $v_1$ and $J_0v_1$; furthermore, since $U^mv_1 \neq 0$ for all $m \geq 0$ and $f_*$ is an $\F[U]$-module homomorphism inducing an isomorphism on $U^{-1}H_*(M)$, we must have $U^mf_*(v_1) \neq 0$. This implies that $f_*(v_1)$ is either $v_1$ or $J_0v_1$. Since $f_*(J_0v_1)= J_0f_*(v_1)$, we see that both $v_1$ and $J_0v_1$ are in the image of $f_*$, and $f_* \co H_{h_1}(M) \rightarrow H_{h_1}(M)$ is a surjective map. 

Now for $1 < i < n$, let us inductively suppose that $f_*\co H_{h_j}(M) \rightarrow H_{h_j}(M)$ is known to be surjective for all $j<i$. In particular, we assume that all of the generators $v_1,J_0v_1, \cdots, v_{i-1}, J_0v_{i-1}$ are in the image of $f_*$. Now consider the element $f_*(v_i)$, which we can write as a linear combination of the $2i$ elements of the form $U^{\frac{h_j-h_i}{2}}v_j$ and $U^{\frac{h_j-h_i}{2}}J_0 v_j$ for $1\leq j \leq i$ (this includes $v_i$ and $J_0v_i$). Since $U^mv_i \neq 0$ for all $m\geq 0$ and $f_*$ is an $\F[U]$-module map inducing an isomorphism on $U^{-1}H_*(M)$, we see that $f_*(v_i)$ cannot only be a sum of terms of the form $U^{\frac{h_j-h_i}{2}}(v_j+J_0v_j)$, but must include at least one element which is not preserved by $J_0$. Choose the maximal number $j=k$ for which such an element appears in $f_*(v_i)$. Up to post-composing with $J_0$ we may assume it is $U^{\frac{h_k-h_i}{2}}v_k$. Now observe that since $U^{\frac{h_i-r_i}{2}}v_i$ is $J_0$-invariant, and $f_*$ commutes with $J_0$, we must have that $U^{\frac{h_i-r_i}{2}} \left( U^{\frac{h_k-h_i}{2}}v_k \right) =U^{\frac{h_k-r_i}{2}}v_k$ is $J_0$-invariant. This implies that $r_k\geq r_i$. However, since by monotonicity $r_k<r_i$ if $k<i$, this implies that $i=k$. So, $f_*(v_i)$ is equal to a sum of $v_i$ and $U$-powers of the elements $v_j$ and $J_0v_j$ for $1\leq j<i$. In particular, since we already know that $v_j$ and $J_0v_j$ are in the image of $f_*$ for $1\leq j<i$, we see that $v_i$ is also in the image of $f_*$. Since $f_*$ is $J_0$-equivariant, $J_0v_i$ is also in the image of $f_*$. We conclude that $f_* \co H_{h_i}(M) \rightarrow H_{h_i}(M)$ is a surjection.

Finally, consider the map $f_*$ in degree $h_n$. If $v_n \neq J_0v_n$, the argument above applies and we are done; otherwise, let $v_n = J_0v_n$. Then the element $f_*(v_n)$ must be fixed by $J_0$ and $U$-nontorsion. By similar logic as above, this implies that if $f_*(v_n)$ is written as a linear combination of $v_n$ and elements $U^{\frac{h_j-h_n}{2}}v_j$ and $U^{\frac{h_j-h_n}{2}}J_0v_j$, this linear combination must contain $v_n$. Since we know that $v_j$ and $J_0v_j$ are in the image of $f_*$ for $1\leq j <n$, this implies that $v_n$ is in the image of $f_*$, and $f_*$ is therefore a surjection.
\end{proof}

\begin{proof}[Proof of Corollary \ref{cor:gradedroot}]
The corollary follows from the definition of the monotone graded root $M'$ associated to a symmetric graded root $M$.
\end{proof}